\documentclass[a4paper,leqno,10pt]{amsart}

\usepackage{epsf,graphicx,epsfig}
\usepackage{amscd}
\usepackage{amsmath,latexsym,amssymb,amsthm}
\usepackage[nospace,noadjust]{cite}
\usepackage{textcomp}
\usepackage{setspace,cite}
\usepackage{lscape,fancyhdr,fancybox}
\usepackage{stmaryrd}
\usepackage[all,cmtip]{xy}
\usepackage{tikz}
\usepackage{cancel}
\usetikzlibrary{shapes,arrows,decorations.markings}
\usepackage{graphicx}

\usepackage{color}
\usepackage{cancel}
\usepackage{url}
\usepackage{enumerate}
\usepackage[mathscr]{euscript}

  \theoremstyle{plain}

\swapnumbers
    \newtheorem{thm}{Theorem}[section]
    \newtheorem{prop}[thm]{Proposition}
     
   \newtheorem{lemma}[thm]{Lemma}
    \newtheorem{corollary}[thm]{Corollary}
    
    \newtheorem{subsec}[thm]{}
\theoremstyle{definition}
    \newtheorem{defn}[thm]{Definition}
        \newtheorem{remark}[thm]{Remark}
\theoremstyle{remark}

\title{}
\author{}
\date{}
\usepackage{amssymb}

\usepackage{hyperref}
\hypersetup{
	colorlinks,
	citecolor=blue,
	filecolor=black,
	linkcolor=blue,
	urlcolor=black
}

\begin{document}
\title{Deformations of associative Rota-Baxter operators}

\author{Apurba Das}
\address{Department of Mathematics and Statistics,
Indian Institute of Technology, Kanpur 208016, Uttar Pradesh, India.}
\email{apurbadas348@gmail.com}

\curraddr{}
\email{}


\begin{abstract}
Rota-Baxter operators and more generally $\mathcal{O}$-operators on associative algebras are important in probability, combinatorics, associative Yang-Baxter equation and splitting of algebras. Using a method of Uchino, we construct an explicit graded Lie algebra whose Maurer-Cartan elements are given by $\mathcal{O}$-operators. This allows us to construct cohomology for an $\mathcal{O}$-operator. This cohomology can also be seen as the Hochschild cohomology of a certain algebra with coefficients in a suitable bimodule. Next, we study linear and formal deformations of an $\mathcal{O}$-operator which are governed by the above-defined cohomology. We introduce  Nijenhuis elements associated with an $\mathcal{O}$-operator which give rise to trivial deformations. As applications, we conclude deformations of weight zero Rota-Baxter operators, associative {\bf r}-matrices and averaging operators.
\end{abstract}

\noindent

\thispagestyle{empty}

\maketitle

\tableofcontents

\vspace{0.2cm}

\section{Introduction}

Deformation theory was first formulated in complex analytic theory around 1960 by the works of Fr\"{o}licher-Nijenhuis \cite{fro-nij}, Kodaira-Spencer \cite{kod-sp} and Kodaira-Nirenberg-Spencer \cite{kns}. After a few years,
 the deformation of algebraic structures was initiated with the seminal work of Gerstenhaber for associative algebras \cite{gers}. It has been extended to various other algebraic structures over the years, including Lie algebras, Leibniz algebras and Poisson algebras \cite{ nij-ric, balav , kont}. In \cite{bala} Balavoine developed a general theory to study the deformation of algebras over a binary quadratic operad.

\medskip

It has been known from Gerstenhaber that the deformation of some algebraic structure is governed by a suitable cohomology theory of the structure. For instance, the deformation of an associative algebra is governed by the classical Hochschild cohomology of the associative algebra \cite{gers, hoch}, while the deformation of a Lie algebra is governed by the Chevalley-Eilenberg cohomology \cite{nij-ric}.

\medskip

Our main objectives in this paper are certain operators on associative algebras. More precisely, we are interested in deformation of Rota-Baxter operators or more generally $\mathcal{O}$-operators on associative algebras. Rota-Baxter operators were first introduced by Baxter in his study of the fluctuation theory in probability \cite{baxter}. This was further developed by Rota \cite{rota} and Cartier \cite{cart}, find their relationship with combinatorics. In particular, the identity of the Rota-Baxter operator can be considered as a possible generalization of the standard shuffle relation \cite{guo-adv}. They were found important applications in the Connes-Kreimer's algebraic approach to the renormalization of quantum field theory \cite{conn}.
Rota-Baxter operators also give rise to dendriform structures introduced by Loday \cite{loday} (see also \cite{fard-guo}). Rota-Baxter operator has been also extended to Lie algebras which are closely related to the solutions of the classical Yang-Baxter equation. See \cite{guo-book} for more details.

\medskip

The notion of $\mathcal{O}$-operators, also known as relative Rota-Baxter operators on associative algebras is a generalization of Rota-Baxter operators in the presence of bimodules. Let $(A, \cdot )$ be an associative algebra and $M$ be an $A$-bimodule. A linear map $T: M \rightarrow A$ is said to be an $\mathcal{O}$-operator on $A$ with respect to the bimodule $M$ if it satisfies
\begin{align*}
T(m) \cdot T(n) = T (m T(n) + T(m)n ), ~~~ \text{ for all } m, n \in M.
\end{align*}
They were first introduced by Uchino \cite{uchino} as an associative analogue of Poisson structures on a manifold. Such an operator gives rise to a dendriform structure on $M$ generalizing the fact from Rota-Baxter operator. Therefore, $M$ inherits an associative structure as well.
In \cite{uchino} Uchino gave a characterization of $\mathcal{O}$-operators as certain Maurer-Cartan elements of a suitable graded Lie algebra (see also \cite{uchino2}).
A class of interesting $\mathcal{O}$-operators are induced from associative {\bf r}-matrices introduced in \cite{aguiar2}. Note that an associative {\bf r}-matrix is an associative analog of classical {\bf r}-matrix.

\medskip

Recently, formal deformation theory of $\mathcal{O}$-operators on Lie algebras has been developed in \cite{tang}. They construct a graded Lie algebra whose Maurer-Cartan elements are $\mathcal{O}$-operators. This allows them to introduce cohomology for an $\mathcal{O}$-operator. Next, they systematically study deformations of an $\mathcal{O}$-operator and show that such deformations are governed by the above cohomology. Our aim in this paper is to develop cohomology and deformation of $\mathcal{O}$-operators on associative algebras. For this, we closely follow the usual deformation approaches and the one developed in \cite{tang}. We construct an explicit graded Lie algebra whose Maurer-Cartan elements are $\mathcal{O}$-operators on associative algebras. This graded Lie algebra is somewhat similar to the one constructed by Uchino \cite{uchino}. Using this, we construct a cochain complex defining the cohomology of an $\mathcal{O}$-operator. This cohomology can be thought of as an associative analogue of Poisson cohomology \cite{bhas}. We give another interpretation of the coboundary operator in terms of the Hochschild coboundary. More precisely, we show that the Hochschild coboundary of the associative algebra $M$ (induced from the dendriform structure on $M$) with coefficients in a certain bimodule structure on $A$ coincides with the above differential up to a sign. Hence they share isomorphic cohomology. 

\medskip

As mentioned earlier, an $\mathcal{O}$-operator on an associative algebra induces a dendriform structure. We find a morphism from the cohomology of an $\mathcal{O}$-operator to the cohomology of the corresponding dendriform algebra.
An $\mathcal{O}$-operator on an associative algebra also gives rise to an $\mathcal{O}$-operator on the commutator Lie algebra. Therefore, it is natural to expect a morphism between the cohomology of an associative $\mathcal{O}$-operator and the cohomology of the corresponding $\mathcal{O}$-operator on Lie algebra. We show that the standard skew-symmetrization process gives rise to such a morphism.

\medskip

Next, we study linear and formal deformations of an $\mathcal{O}$-operator on an associative algebra. In such a theory, it is expected to deform the algebra, the bimodule and the $\mathcal{O}$-operator. See \cite{sheng} where deformations of algebras with bimodules are studied. However, like \cite{tang}, we restrict ourself only to deform the $\mathcal{O}$-operator.
Linear terms (coefficient of the parameter $t$) of such deformations are $1$-cocycle in the cohomology of the $\mathcal{O}$-operator, called infinitesimals. Moreover, equivalent deformations have cohomologous infinitesimals. We also introduce Nijenhuis elements associated to an $\mathcal{O}$-operator which give rise to trivial linear deformations.
The extension problem of a finite order deformation to the next order is also discussed. We show that the vanishing of the second cohomology allows one to extend a finite order deformation to a deformation of the next order. A deformation of an $\mathcal{O}$-operator induces a deformation of the corresponding dendriform structure on $M$. Note that the deformation of a dendriform structure has been recently studied in \cite{das}.

\medskip

As it is mentioned earlier that Rota-Baxter operators of weight $0$ and associative {\bf r}-matrices are special cases of $\mathcal{O}$-operators. Therefore, we may study the deformation of such structures as particular cases of deformation of $\mathcal{O}$-operators. In the case of deformation of a Rota-Baxter operator, we only state the main results as they are completely analogous to the case of $\mathcal{O}$-operator. We also define deformation of an associative {\bf r}-matrix which is compatible with the deformation of the corresponding $\mathcal{O}$-operator. In this regard, we obtained some results about morphism between {\bf r}-matrices,  and morphism between corresponding $\mathcal{O}$-operators. Finally, we discuss left (right) averaging operators on associative algebras \cite{pei-guo}, their cohomology and deformations.

\medskip

The paper is organized as follows. In the next section (Section \ref{sec2}), we recall Rota-Baxter operators and $\mathcal{O}$-operators on associative algebras. We also describe morphism between $\mathcal{O}$-operators. Next, we construct an explicit graded Lie algebra whose Maurer-Cartan elements are precisely $\mathcal{O}$-operators. Using this, one can define cohomology for an $\mathcal{O}$-operator. In Section \ref{sec3}, we show that this cohomology is isomorphic to the Hochschild cohomology of $M$ with suitable bimodule structure on $A$. We also compare the cohomology of an associative $\mathcal{O}$-operator with the cohomology of the $\mathcal{O}$-operator on the commutator Lie algebra.
Linear and formal deformations of an $\mathcal{O}$-operator are discussed in Section \ref{sec4}. Nijenhuis elements, trivial deformations, extension of deformations are also discussed in this section. Finally, in Section \ref{sec5}, we apply the above cohomology and deformation theory to Rota-Baxter operators, associative {\bf r}-matrices and left (right)  averaging operators.

\medskip

All vector spaces, linear maps, tensor products are over a field $\mathbb{K}$ of characteristic $0$. The elements of the vector space $A$ are usually denoted by $a, b, c, \ldots $ and the elements of $M$ by $m, n, u, v, u_1, u_2, \ldots.$ We denote the permutation group of $n$ elements by $\mathbb{S}_n$.

\section{Rota-Baxter operators and $\mathcal{O}$-operators}\label{sec2}

Our aim in this section is to recall Rota-Baxter operators, $\mathcal{O}$-operators on associative algebras and their morphisms. We also recall dendriform structures induced from $\mathcal{O}$-operators. Finally, we construct a graded Lie algebra with explicit graded Lie bracket whose Maurer-Cartan elements are given by $\mathcal{O}$-operators. This allows us to define cohomology for an $\mathcal{O}$-operator. Our main references are \cite{ fard-guo, loday , uchino }.

\begin{defn}Let $(A, \cdot)$ be an associative algebra. A linear map $R : A \rightarrow A$ is said to be a Rota-Baxter operator of weight $\lambda \in \mathbb{K}$ if it satisfies
\begin{align*}
R(a) \cdot R(b) = R \big( a \cdot R(b) + R(a) \cdot b + \lambda a \cdot b \big), ~~~ \text{ for all } a, b \in A.
\end{align*}
\end{defn}

The notion of $\mathcal{O}$-operators (also called generalized Rota-Baxter operators or relative Rota-Baxter operators) are a generalization of Rota-Baxter operators in the presence of arbitrary bimodule.
Let $(A, \cdot)$ be an associative algebra and $M$ be a bimodule over $A$. That is, there are linear maps $l: A \otimes M \rightarrow M, ~ (a, m ) \mapsto l(a, m)$ and $r : M \otimes A \rightarrow M, ~ (m, a) \mapsto r (m, a)$ satisfying
\begin{align*}
l (a \cdot b , m ) = l (a, l (b, m )), \quad  l ( a, r (m, b)) = r ( l (a, m), b) ~~~~ \text{ and } ~~~~ r ( m, a \cdot b ) = r ( r (m, a ), b),
\end{align*}
for all $a, b \in A$ and $m \in M.$ Thus, for each $a \in A$, there are maps $l_a : M \rightarrow M, ~ m \mapsto l (a, m)$ and $r_a : M \rightarrow M, ~ m \mapsto r(m, a)$. We will frequently use these notations in the rest of the paper. Sometimes we also write $am$ instead of $l(a, m)$ and $m a$ instead of $r (m, a)$ when there are no confusions.

It follows that the associative algebra $A$ is a bimodule over itself with the left and right actions are given by the multiplication of $A$. We call this bimodule as adjoint bimodule. The maps $l_a$ and $r_a$ for this bimodule are denoted by $\text{ad}^l_a$ and $\text{ad}^r_a$, respectively. The dual space $A^*$ also carries an $A$-bimodule (called coadjoint bimodule) structure with 
\begin{align*}
l (a, f ) (b ) =  f ( b \cdot a) ~~ \text{ and } ~~ r (f, a ) (b) = f ( a \cdot b),
\end{align*}
for $a, b \in A$ and $ f \in A^*.$ The maps $l_a$ and $r_a$ for this bimodule are respectively denoted by $\text{ad}^{*l}_a$ and $\text{ad}^{* r}_a$. Adjoint and coadjoint bimodules play a central role to study deformation theory of Rota-Baxter operators and associative {\bf r}-matrices.

Given an associative algebra $(A, \cdot)$ and a bimodule $M$, the vector space $A \oplus M$ carries an associative structure with product given by
\begin{align*}
(a, m) \cdot (b, n ) = ( a \cdot b , an + mb ).
\end{align*}
This is called the semi-direct product of $A$ with $M$.

\begin{defn}
Let $(A, \cdot)$ be an associative algebra and $M$ be a bimodule over it. An $\mathcal{O}$-operator on $A$ with respect to the bimodule $M$ is given by a linear map $T : M \rightarrow A$ that satisfies
\begin{align*}
T(m) \cdot T (n ) = T ( m  T(n) + T(m) n ), ~~~ \text{ for all } m, n \in M.
\end{align*}
\end{defn}

When $M = A$ with the adjoint bimodule structure, an $\mathcal{O}$-operator is nothing but a Rota-Baxter operator on $A$ of weight $\lambda = 0.$ Thus an $\mathcal{O}$-operator is a generalization of Rota-Baxter operator.

In \cite{uchino} Uchino gave the following characterization of an $\mathcal{O}$-operator.
\begin{prop}
A linear map $T : M \rightarrow A$ is an $\mathcal{O}$-operator on $A$ with respect to the $A$-bimodule $M$ if and only if the graph of $T$,
\begin{align*}
\mathrm{Gr}(T) = \{ (T(m), m ) |~ m \in M \}
\end{align*}
is a subalgebra of the semi-direct product algebra $A \oplus M$.
\end{prop}

This characterization of an $\mathcal{O}$-operator can be thought of as an associative analogue of the fact that a bivector field $\pi \in \Gamma (\wedge^2 TM)$ on a manifold $M$ is a Poisson tensor if and only if the graph of the bundle map $\pi^\sharp : T^*M \rightarrow TM$ is a Dirac structure on $M$ \cite{courant}.

Another characterization of an $\mathcal{O}$-operator can be given in terms of Nijenhuis operator on associative algebras.

\begin{defn}(\cite{grab})
Let $(A , \cdot)$ be an associative algebra. A linear map $N : A \rightarrow A$ is said to be a Nijenhuis operator if its Nijenhuis torsion vanishes, i.e.
\begin{align*}
N(a) \cdot N(b) = N ( Na \cdot b + a \cdot Nb - N (a \cdot b) ), ~~ \text{ for all } a, b \in A.
\end{align*}
\end{defn}
It follows that the bilinear operation $\cdot_N : A \otimes A \rightarrow A$ given by $a \cdot_N b = Na \cdot b + a \cdot Nb - N (a \cdot b)$ defines a new associative multiplication on $A$, and $N$ becomes an algebra morphism from $(A, \cdot_N)$ to $(A, \cdot).$

Nijenhuis operator on associative algebra is an associative analogue of Nijenhuis operator on Lie algebra or a manifold.

\begin{prop}\label{o-nij}
A linear map $T: M \rightarrow A$ is an $\mathcal{O}$-operator on $A$ with respect to the bimodule $M$ if and only if $N_T = \begin{pmatrix} 
0 & T \\
0 & 0 
\end{pmatrix} : A \oplus M \rightarrow A \oplus M$ is a Nijenhuis operator on the semi-direct product algebra $A \oplus M$.
\end{prop}

\medskip

Next, we recall dendriform structures which were first introduced by Loday in his study of the periodicity phenomenons in algebraic $K$-theory \cite{loday}. Dendriform structures pay very much attention in the last 20 years due to its connection with Rota-Baxter algebras, shuffle algebras and combinatorics \cite{aguiar, uchino, guo-adv}.

\begin{defn}
A dendriform algebra is a vector space $D$ together with two linear maps $\prec, \succ : D \otimes D \rightarrow D$ satisfying the following three identities
\begin{align}
( a \prec b ) \prec c =~& a \prec ( b \prec c + b \succ c ), \label{dend-def-1}\\
( a \succ b ) \prec c =~& a \succ ( b \prec c),\\
( a \prec b + a \succ b ) \succ c =~& a \succ ( b \succ c), \label{dend-def-3}
\end{align}
for all $a , b, c \in D$.
\end{defn}

It follows from (\ref{dend-def-1})-(\ref{dend-def-3}) that the new operation $\star : D \otimes D \rightarrow D$ defined by $a \star b = a \prec b ~+~ a \succ b$ turns out to be associative. Thus, a dendriform algebra can be thought of as a splitting of an associative algebra.

An $\mathcal{O}$-operator has an underlying dendriform structure \cite{uchino}.

\begin{prop}\label{o-dend}
Let $T : M \rightarrow A$ be an $\mathcal{O}$-operator on $A$ with respect to the bimodule $M$. Then the vector space $M$ carries a dendriform structure with
\begin{align*}
m \prec n = m T(n) ~~~ \text{ and } ~~~  m \succ n = T(m ) n, ~~~ \text{ for all } m , n \in M.
\end{align*}
\end{prop}

\medskip

Next, we study morphism between $\mathcal{O}$-operators.
Let $(A, \cdot_A )$ be an associative algebra and $M$ an $A$-bimodule, and $(B, \cdot_B)$ be an associative algebra with $N$ an $B$-bimodule. Suppose $T : M \rightarrow A$ is an $\mathcal{O}$-operator on $A$ with respect to the $A$-bimodule $M$ and $T' : N \rightarrow B$ is an $\mathcal{O}$-operator on $B$ with respect to the bimodule $N$.

\begin{defn}\label{o-op-map}
A  morphism of $\mathcal{O}$-operators from $T$ to $T'$ consists of a pair $(\phi, \psi)$ of an algebra morphism $\phi : A \rightarrow B$ and a linear map $\psi : M \rightarrow N$ satisfying
\begin{align}
T' \circ \psi =~& \phi \circ T, \label{eq1}\\
\phi (a) \psi (m ) =~& \psi ( am ), \label{eq2}\\
\psi (m) \phi (a ) =~& \psi (m a ),  \label{eq3}
\end{align}
for all  $a \in A$ and $ m \in M$.

It is called an isomorphism if $\phi$ and $\psi$ are both linear isomorphisms.
\end{defn}

The proof of the following result is straightforward and we omit the details.

\begin{prop}
A pair of linear maps $(\phi : A \rightarrow B, ~ \psi : M \rightarrow N)$ is a morphism of $\mathcal{O}$-operators from $T$ to $T'$ if and only if
\begin{align*}
\mathrm{Gr} ((\phi, \psi)) := \big\{ \big( (a, m), (\phi(a), \psi (m)) \big) | ~ a \in A, m \in M \big\} \subset (A \oplus M) \oplus (B \oplus N)
\end{align*}
is a subalgebra, where $A \oplus M$ and $B \oplus N$ are equipped with semi-direct product algebra structures.
\end{prop}

\begin{prop}
Let $T$ be an $\mathcal{O}$-operator on $A$ with respect to a bimodule $M$ and $T'$ be an $\mathcal{O}$-operator on $B$ with respect to a bimodule $N$. If $(\phi, \psi)$ is a morphism from $T$ to $T'$, then $\psi : M \rightarrow N$ is a morphism between induced dendriform structures.
\end{prop}

\begin{proof}
For all $m, m' \in M$, we have
\begin{align*}
\psi (m \prec_M m') = \psi (m T(m')) \stackrel{(\ref{eq3})}{=} \psi (m)  (\phi T (m'))
\stackrel{(\ref{eq1})}{=} \psi (m )  (T' \psi (m') )
 = \psi (m) \prec_N \psi (m'). 
\end{align*}
Similarly,
\begin{align*}
\psi (m \succ_M m' ) = \psi (T(m) m' )  \stackrel{(\ref{eq2})}{=} \phi T(m) \psi (m' ) \stackrel{(\ref{eq1})}{=} T' \psi (m) \psi (m') = \psi (m) \succ_N \psi (m').
\end{align*}
Hence the result follows.
\end{proof}

In the rest of the paper, we will be most interested in morphism between $\mathcal{O}$-operators on the same algebra with respect to the same bimodule.

\subsection{A graded Lie algebra}
In \cite{uchino} Uchino considers a graded Lie algebra associated to an associative algebra and a bimodule over it. An $\mathcal{O}$-operator can be characterized by certain solutions of the corresponding Maurer-Cartan equation.

Here we follow the result of Uchino and the derived bracket construction of Voronov \cite{voro} to construct an explicit graded Lie algebra whose Maurer-Cartan elements are $\mathcal{O}$-operators. This construction is somewhat similar to Uchino but more helpful to study deformation theory of $\mathcal{O}$-operators.

Recall that, in \cite{gers2} Gerstenhaber construct a graded Lie algebra structure on the graded space of all multilinear maps on a vector space $V$. More precisely, for each $n \geq 0$, he defined $\mathfrak{g}^n = \text{Hom} ( V^{\otimes n+1}, V)$ and a graded Lie bracket on $\oplus_n \mathfrak{g}^n$ by
\begin{align*}
[f, g]  =   f \circ g - (-1)^{mn} g \circ f 
\end{align*}
where 
\begin{align*}
(f \circ g ) (v_1, \ldots, v_{m+n+1} ) =  \sum_{i = 1}^{m+1} (-1)^{(i-1)n}~f ( v_1, \ldots, v_{i-1}, g ( v_i, \ldots, v_{i+n}), \ldots, v_{m+n+1}),
\end{align*}
for $f \in \mathfrak{g}^m, ~g \in \mathfrak{g}^n$ and $v_1, \ldots, v_{m+n+1} \in V$.

Let $A$ and $M$ be two vector spaces equipped with maps $\mu : A^{\otimes 2} \rightarrow A$, $l : A \otimes M \rightarrow M$ and $r : M \otimes A \rightarrow M$. Consider the graded Lie algebra structure on $\mathfrak{g} = \oplus_{n \geq 0} \text{Hom} ((A \oplus M)^{\otimes n+1}, A \oplus M )$ associated to the direct sum vector space $V = A \oplus M$. 
Observe that the elements $\mu , l, r \in \mathfrak{g}^1 = \text{Hom} ((A \oplus M)^{\otimes 2}, A \oplus M)$. Therefore, $\mu + l + r \in \mathfrak{g}^1.$

\begin{prop}\label{mlr}
The product $\mu$ defines an associative structure on $A$ and $l, r$ defines an $A$-bimodule structure on $M$ if and only if $\mu + l + r \in \mathfrak{g}^1$ is a Maurer-Cartan element in $\mathfrak{g}.$
\end{prop}

\begin{proof}
The element $\mu + l + r$ is a Maurer-Cartan element if and only if $(\mu + l + r) \circ ( \mu + l + r ) = 0$. This is same as
\begin{align*}
(\mu + l + r) &\big( (\mu+ l + r) ((a_1 , m_1), (a_2 , m_2)), (a_3, m_3) \big) \\
&= (\mu + l + r) \big(  (a_1, m_1), (\mu+ l + r ) ((a_2 , m_2), (a_3 , m_3)) \big)
\end{align*}
or equivalently,
\begin{align*}
\big( (a_1 a_2) a_3,~  (a_1 a_2) m_3 + (a_1 m_2) a_3 + (m_1 a_2 ) a_3 \big) = \big(  a_1 (a_2 a_3),~ a_1 (a_2 m_3) + a_1 (m_2 a_3 ) + (m_1 a_2 ) a_3   \big).
\end{align*}
This holds if and only if $\mu$ defines an associative structure on $A$ and $l, r$ defines an $A$-bimodule structure on $M$.
\end{proof}

Let $A$ be an associative algebra and $l, r$ defines an $A$-bimodule structure on $M$.
By the above proposition, the graded Lie algebra $ (\oplus_{n \geq 0} \text{Hom} ((A \oplus M)^{\otimes n+1}, A \oplus M ), [~, ~] )$ together with the differential $d_{\mu + l + r } = [ \mu + l + r, ~ ]$ is a dgLa.
Moreover, it is easy to see that the graded subspace $\oplus_{n \geq 0} \text{Hom} (M^{\otimes n+1}, A )$ is an abelian subalgebra.
Therefore, we are now in a position to apply the derived bracket construction of Voronov \cite{voro} (see also \cite{uchino3}) to get a graded Lie algebra structure on $\oplus_{n \geq 1} \text{Hom} (M^{\otimes n}, A )$ with bracket 
\begin{align}
\llbracket P, Q \rrbracket := (-1)^m [ [ \mu + l + r , P ], Q ],
\end{align}
for $P \in \text{Hom}(M^{\otimes m}, A), ~ Q \in \text{Hom}(M^{\otimes n}, A)$, $m , n \geq 1$.
Explicitly, the bracket is given by
\begin{align}\label{derived-bracket}
&\llbracket P, Q \rrbracket  (u_1, \ldots, u_{m+n}) \\
&= \sum_{i=1}^m (-1)^{(i-1) n} P (u_1, \ldots, u_{i-1}, Q (u_i, \ldots, u_{i+n-1}) u_{i+n}, \ldots, u_{m+n} ) \nonumber\\
~&- \sum_{i=1}^m (-1)^{in} P (u_1, \ldots, u_{i-1}, u_i Q (u_{i+1}, \ldots, u_{i+n}), u_{i+n+1}, \ldots, u_{m+n}) \nonumber\\
&- (-1)^{mn} \bigg\{ \sum_{i=1}^n (-1)^{(i-1)m}~ Q (u_1, \ldots, u_{i-1}, P (u_i, \ldots, u_{i+m-1}) u_{i+m}, \ldots, u_{m+n}) \nonumber\\
&- \sum_{i=1}^n (-1)^{im}~ Q (u_1, \ldots, u_{i-1}, u_i P (u_{i+1}, \ldots, u_{i+m}) , u_{i+m+1}, \ldots, u_{m+n} ) \bigg\} \nonumber\\
&+ (-1)^{mn} \big[ P(u_1, \ldots, u_m) \cdot Q (u_{m+1}, \ldots, u_{m+n} ) - (-1)^{mn}~ Q (u_1, \ldots, u_n) \cdot P ( u_{n+1}, \ldots, u_{m+n} ) \big], \nonumber
\end{align}
for $u_1, \ldots, u_{m+n} \in M.$ We may extend the graded Lie bracket to $\oplus_{n \geq 0} \text{Hom} (M^{\otimes n}, A )$ as follows.
For $P \in \text{Hom}(M^{\otimes m}, A)$ and $a \in A$, the bracket is 
\begin{align}\label{derived-bracket-0}
\llbracket P, a \rrbracket (u_1, \ldots, u_m ) =~& \sum_{i=1}^m P (u_1, \ldots, u_{i-1}, a u_i - u_i a , u_{i+1}, \ldots, u_{m} ) \\
~~&+ P (u_1, \ldots, u_m) \cdot a - a \cdot P (u_1, \ldots, u_m) \nonumber
\end{align}
and for $a, b \in A$, we define $\llbracket a, b \rrbracket$ to be the commutator Lie bracket $[a, b ]_C = a \cdot b - b \cdot a.$

Thus, for any $T, T' \in \text{Hom} (M, A)$, we have from (\ref{derived-bracket}) that
\begin{align}\label{t-t}
\llbracket T, T' \rrbracket (u, v) = T (T'(u)v ) + T(u T'(v)) + T' (T(u)v) + T' (u T(v)) - T(u) \cdot T'(v) - T'(u) \cdot T(v).
\end{align}

For any $n \geq 0$, we define $C^n (M, A) := \mathrm{Hom} (M^{\otimes n}, A)$ and consider the graded vector space 
$C^\bullet (M, A ) = \oplus_{n \geq 0}  C^n (M, A) =   \oplus_{n \geq 0}  \mathrm{Hom} (M^{\otimes n}, A)$. Thus, we obtain the following.

\begin{thm}\label{new-gla}
The graded vector space $C^\bullet (M, A ) = \oplus_{n \geq 0} \mathrm{Hom} (M^{\otimes n}, A)$ together with the above defined bracket $\llbracket ~, ~\rrbracket$ forms a graded Lie algebra. A linear map $T : M \rightarrow A$ is an $\mathcal{O}$-operator on $A$ with respect to the bimodule $M$ if and only if $T \in C^1 (M, A)$ is a Maurer-Cartan element in $(C^\bullet (M, A ), \llbracket~,~\rrbracket )$, i.e. $T$ satisfies $\llbracket T, T \rrbracket = 0$.
\end{thm}

Therefore, from the general principle of Maurer-Cartan elements, we have the following.

\begin{thm}\label{new-gla-2}
Let $T : M \rightarrow A$ be an $\mathcal{O}$-operator on $A$ with respect to the bimodule $M$. Then $T$ induces a differential $d_T = \llbracket T, ~ \rrbracket$ which makes the graded Lie algebra $( C^\bullet (M, A ), \llbracket ~,~\rrbracket)$ into a dgLa.

Moreover, for any linear map $T' : M \rightarrow A$, the sum $T + T'$ is still an $\mathcal{O}$-operator if and only if $T'$ is a Maurer-Cartan element in the dgLa $( C^\bullet (M, A ), \llbracket ~,~ \rrbracket, d_T)$, in other words,
\begin{align*}
\llbracket T + T' , T + T' \rrbracket = 0  ~ \Longleftrightarrow ~ d_T (T') + \frac{1}{2} \llbracket T', T' \rrbracket = 0.
\end{align*}
\end{thm}

It follows from the above theorem that if $T$ is an $\mathcal{O}$-operator, then $(C^\bullet (M, A), d_T = \llbracket T, ~ \rrbracket)$ is a cochain complex. Its cohomology is called the cohomology of the $\mathcal{O}$-operator $T$. If an $\mathcal{O}$-operator serves as an associative analogue of Poisson structure, then the above cohomology of $\mathcal{O}$-operator serves as the associative analogue of Poisson cohomology. 

Here we will not use any notation to denote this cohomology, as in the next section, we interpret this cohomology as the Hochschild cohomology of a certain algebra with coefficients in a suitable bimodule. Then we will use the usual notation for Hochschild cohomology to denote the cohomology of an $\mathcal{O}$-operator.




\medskip

\section{Cohomology of $\mathcal{O}$-operators as Hochschild cohomology}\label{sec3}

Our aim in this section is to show that the cohomology of an $\mathcal{O}$-operator can also be described as the Hochschild cohomology of a certain associative algebra with a suitable bimodule. We also relate the cohomology of an $\mathcal{O}$-operator on an associative algebra with the cohomology of the corresponding $\mathcal{O}$-operator on the commutator Lie algebra.

Let $T : M \rightarrow A$ be an $\mathcal{O}$-operator on $A$ with respect to the bimodule $M$. Then by Proposition \ref{o-dend} the vector space $M$ carries an associative algebra structure with the product
\begin{align}\label{m-ass}
m \star n = m T(n ) + T(m) n, ~~~ \text{ for } m , n \in M.
\end{align}

The following result has also been mentioned in \cite{uchino}.
\begin{lemma}\label{new-rep-o}
Let $T: M \rightarrow A$ be an $\mathcal{O}$-operator on $A$ with respect to the bimodule $M$. Define
\begin{align*}
l_T : M \otimes A \rightarrow A, ~~~~~ l_T (m, a) = T(m) \cdot a - T (m a),\\
r_T : A \otimes M \rightarrow A, ~~~~~ r_T (a, m) = a \cdot T(m) - T (a m),
\end{align*}
for $m \in M$ and $a \in A$. Then $l_T, ~ r_T$ defines an $M$-bimodule structure on $A$.
\end{lemma}

\begin{proof}
For any $m, n \in M$ and $a \in A$, we have
\begin{align*}
l_T (m \star n, a) =~& T(m \star n) \cdot a - T ((m \star n) a) \\
=~& T(m) \cdot T(n) \cdot a - T ((mT(n) + T(m)n)a)
\end{align*}
and
\begin{align*}
l_T (m, l_T (n, a) ) =~& l_T (m, T(n) \cdot a - T(na) ) \\
=~& T(m) \cdot (T(n) \cdot a - T(na)) - T (m (T(n) \cdot a - T(na) )) \\
=~& T(m) \cdot T(n) \cdot a - T (\cancel{m T(na)} + T(m) na) - T (m (T(n) \cdot a)) + \cancel{T (m T(na))}.
\end{align*}
Hence we have $l_T (m \star n, a) = l_T (m, l_T(n,a))$. Similarly, we can prove that $l_T (m, r_T (a, n)) = r_T (l_T(m,a), n)$ and $r_T (a, m \star n) = r_T (r_T (a, m), n ).$ Hence the proof.
\end{proof}

\begin{remark}
The above bimodule structure of the associative algebra $(M , \star)$ on the vector space $A$ can also be justified in the following way. Note that from Proposition \ref{o-nij}, the vector space $A \oplus M$ has a deformed associative product
\begin{align*}
(a, m) \cdot_{N_T} (b, n) = N_T (a,m) \cdot (b, n) ~+~ (a, m) \cdot N_T (b, n) ~-~ N_T ((a,m) \cdot (b,n)). 
\end{align*}
It is easy to verify that this product is same as
$(l_T (m, b) + r_T (a, n),~ m \star n).$
\end{remark}

By Lemma \ref{new-rep-o} we obtain an $M$-bimodule structure on the vector space $A$. Therefore, we may consider the corresponding Hochschild cohomology of $M$ with coefficients in $A$. More precisely, we define
\begin{align*}
C^n(M, A) := \text{Hom} ( M^{\otimes n}, A), ~~ \text{ for } n \geq 0
\end{align*}
and the differential given by
\begin{align}\label{zero-diff}
d_H (a) (m) = l_T (m, a) - r_T (a, m) = T(m) \cdot a - T (ma) - a \cdot T(m) + T (am),  ~~~ \text{ for } a \in A = C^0 (M, A)
\end{align}
and
\begin{align*}
(d_H  f) (u_1, \ldots, u_{n+1}) =~& T ( u_1) \cdot f (u_2, \ldots, u_{n+1})  - T (u_1 f(u_2, \ldots, u_{n+1})) \\
~&+ \sum_{i=1}^n (-1)^i ~ f (u_1, \ldots, u_{i-1}, u_i T( u_{i+1}) + T(u_i) u_{i+1}, \ldots, u_{n+1}) \\
~&+ (-1)^{n+1}~  f (u_1, \ldots, u_n ) \cdot T(u_{n+1})  - (-1)^{n+1} T ( f(u_1, \ldots, u_n) u_{n+1} ).
\end{align*}
We denote the group of $n$-cocycles by $Z^n (M, A)$ and the group of $n$-coboundaries by $B^n (M, A)$.
The corresponding cohomology groups are defined by $H^n (M, A) = Z^n (M, A) / B^n (M, A)$, $n \geq 0$.

It follows from the above definition that
\begin{align*}
H^0 (M, A) =~& \{ a \in A |~ d_H (a) = 0 \} \\
=~& \{ a \in A |~ a \cdot T(m) - T(m) \cdot a = T(am -ma ), ~ \forall m \in M \}.
\end{align*}
From this definition, it is easy to see that if $a, b \in H^0 (M, A)$, then their commutator $[a, b]_C := a \cdot b - b \cdot a$ is also in $H^0 (M, A)$. This shows that $H^0 (M, A)$ has a Lie algebra structure induced from that of $A$.

Note that a linear map $f : M \rightarrow A$ (i.e. $f \in C^1 (M, A)$) is closed if it satisfies
\begin{align}\label{1-coc}
T(u) \cdot f (v) + f (u) \cdot T(v) - T (u f(v) + f(u) v ) - f (uT(v) + T(u)v ) = 0,
\end{align}
for all $u, v \in M.$

For an $\mathcal{O}$-operator $T$ on $A$ with respect to the bimodule $M$, we get two coboundary operators $d_T = \llbracket T, ~\rrbracket$ and $d_H$ on the same graded vector space $C^\bullet (M, A) = \oplus_{n \geq 0} C^n (M, A)$. The following proposition relates the above two coboundary operators.

\begin{prop}
Let $T : M \rightarrow A$ be an $\mathcal{O}$-operators on $A$ with respect to the $A$-bimodule $M$. Then the two coboundary operators are related by
\begin{align*}
d_T f = (-1)^n~ d_H f , ~~~~ \text{ for } f \in C^n (M, A).
\end{align*}
\end{prop}

\begin{proof}
For any $f \in C^n (M, A) = \text{Hom}(M^{\otimes n}, A )$ and $u_1, \ldots, u_{n+1} \in M$, we have from (\ref{derived-bracket}) that
\begin{align*}
(d_T f) (u_1, \ldots, u_{n+1}) =~& \llbracket T, f \rrbracket \\
=~& T (f(u_1, \ldots, u_n) u_{n+1}) - (-1)^n ~ T (u_1 f(u_2, \ldots, u_{n+1})) \\
-~& (-1)^n \bigg\{ \sum_{i=1}^n (-1)^{i-1}~ f (u_1, \ldots, u_{i-1}, T(u_i) u_{i+1} , u_{i+2}, \ldots, u_{n+1}) \\
-~& \sum_{i=1}^n (-1)^i f (u_1, \ldots, u_{i-1}, u_i T(u_{i+1}), u_{i+2}, \ldots, u_{n+1} ) \bigg\} \\
+~& (-1)^n ~T(u_1) \cdot f (u_2, \ldots, u_{n+1}) - f (u_1, \ldots, u_n) \cdot T(u_{n+1})\\
=~& (-1)^n ~ (d_H f) (u_1, \ldots, u_{n+1}).
\end{align*}
The same holds true when $f = a \in A$. Compare (\ref{derived-bracket-0}) with (\ref{zero-diff}). Hence the proof.
\end{proof}

This shows that the cohomology of the either complex $(C^\bullet (M, A), d_T)$ or $(C^\bullet (M, A), d_H)$ are isomorphic. Thus, we may use the same notation $H^\bullet (M, A)$ to denote the cohomology of an $\mathcal{O}$-operator.

\subsection{Relation with the cohomology of dendriform algebra}\label{subsec-dend-coho} Cohomology of dendriform algebras with trivial coefficients was first defined by Loday in \cite{loday}. The explicit description of the cohomology with coefficients and relation with deformation theory has been given in \cite{das}.

Let $C_n$ be the set of first $n$ natural numbers. For convenience, we denote the elements of $C_n$ by $\{ [1], \ldots, [n] \}.$ For any vector space $D$, consider the collection $\{ \mathcal{O}(n) \}_{n \geq 1}$ of vector spaces  where
\begin{align*}
\mathcal{O}(n) := \mathrm{Hom}( \mathbb{K}[C_n] \otimes D^{\otimes n}, D), ~~\text{ for } n \geq 1.
\end{align*}
Then the collection of vector spaces $\{ \mathcal{O}(n) \}_{n \geq 1}$ with the partial compositions

\medskip

$(f \circ_i g)([r]; a_1, \ldots, a_{m+n-1}) =$
\begin{align}\label{dend-partial-comp}
\begin{cases}
f ([r]; a_1, \ldots, a_{i-1}, g ([1]+ \cdots + [n]; a_i, \ldots, a_{i+n-1}), \ldots, a_{m+n-1}) ~& \text{ if } 1 \leq r \leq i-1 \\
f ([i]; a_1, \ldots, a_{i-1}, g ([r-i+1]; a_i, \ldots, a_{i+n-1}), \ldots, a_{m+n-1}) ~& \text{ if } i \leq r \leq i+ n-1\\
f ([r-n+1]; a_1, \ldots, a_{i-1}, g ([1]+ \cdots + [n]; a_i, \ldots, a_{i+n-1}), \ldots, a_{m+n-1}) ~& \text{ if } i+ n \leq r \leq m + n -1,
\end{cases}
\end{align}
for $f \in \mathcal{O}(m), ~ g \in \mathcal{O}(n),~ 1 \leq i \leq m$ and $[r] \in C_{m+n-1}$; forms a non-symmetric operad. Therefore, there is a graded Lie algebra structure on the graded vector space $\mathcal{O}(\bullet +1 ) = \oplus_{n \geq 0} \mathcal{O}(n+1)$ given by
\begin{align*}
\llceil f, g \rrceil = \sum_{i=1}^{m+1} (-1)^{(i-1)n} f \circ_i g ~-~ (-1)^{mn} \sum_{i=1}^{n+1} (-1)^{(i-1)m} g \circ_i f,
\end{align*}
for $f \in \mathcal{O}(m+1)$ and $g \in \mathcal{O}(n+1)$. More generally, if $(D, \prec, \succ)$ is a dendriform algebra, then the element $\pi \in \mathcal{O}(2)$ defined by
\begin{align*}
\pi ([1]; a, b ) = a \prec b ~~~~ \text{ and } ~~~~ \pi ([2]; a, b ) = a \succ b
\end{align*}
satisfies $\llceil \pi, \pi \rrceil = 0$, i.e. $\pi$ defines a Maurer-Cartan element in the above graded Lie algebra. Hence $\pi$ induces a differential $\delta_\pi : \mathcal{O}(n) \rightarrow \mathcal{O}(n+1)$ given by $\delta_\pi (f) := (-1)^{n-1} \llceil \pi, f \rrceil$, for $f \in \mathcal{O}(n)$. The corresponding cohomology groups are called the cohomology of the dendriform algebra $(D, \prec, \succ)$ and they are denoted by $H^n_{\mathrm{dend}}(D, D)$. \\

Let $T: M \rightarrow A$ be an $\mathcal{O}$-operator on $A$ with respect to a bimodule $M$. Consider the dendriform algebra structure on $M$ induced by $T$. We denote the corresponding Maurer-Cartan element by $\pi_T \in \mathrm{Hom}( \mathbb{K}[C_2 ] \otimes M^{\otimes 2}, M)$. For each $n \geq 0$, we define a map
\begin{align*}
\Theta_n : \mathrm{Hom}( M^{\otimes n}, A) \rightarrow \mathrm{Hom}( \mathbb{K}[C_{n+1}] \otimes M^{\otimes n+1}, M) ~~~\text{ by }
\end{align*}
\begin{align*}
\Theta_n (P ) ([r]; u_1, u_2, \ldots, u_{n+1}) = \begin{cases} (-1)^{n+1} ~u_1 P(u_2, \ldots u_{n+1})  ~~~& \text{if }~ r = 1 \\
0 ~~~& \text{if } ~2 \leq r \leq n \\
P(u_1, \ldots, u_n ) u_{n+1} ~~~& \text{if }~ r = n + 1. \end{cases}
\end{align*}
With these notations, we have the following.

\begin{lemma}
The maps $\{ \Theta_n \}$ preserve corresponding graded Lie brackets, i.e.
for $P \in \mathrm{Hom}( M^{\otimes m}, A)$ and $Q \in \mathrm{Hom}( M^{\otimes n}, A),$
\begin{align*}
\llceil \Theta_m (P), \Theta_n (Q) \rrceil = \Theta_{m + n } ( \llbracket P, Q \rrbracket).
\end{align*}
\end{lemma}

\begin{proof}
For $u_0, u_1, \ldots, u_{m+n} \in M$, we have
\begin{align*}
&\llceil \Theta_m (P), \Theta_n (Q) \rrceil ([1]; u_0, u_1, \ldots, u_{m+n}) \\
&= \big( \sum_{i=1}^{m+1} (-1)^{(i-1)n} \Theta_m (P) \circ_i \Theta_n (Q) ~-~ (-1)^{mn} \sum_{i=1}^{n+1} (-1)^{(i-1)m} \Theta_n (Q) \circ_i \Theta_m (P) \big) ([1]; u_0, \ldots, u_{m+n}) \\
&= \Theta_m (P) ( [1]; \Theta_n (Q) ([1]; u_0, \ldots, u_n), u_{n+1}, \ldots, u_{m+n} )\\
&+ \sum_{i=1}^m (-1)^{in} \Theta_m (P) ([1]; u_0, \ldots, u_{i-1}, \Theta_n (Q) ([1] + \cdots + [n+1]; u_i, \ldots, u_{i+n}), u_{i+n+1}, \ldots, u_{m+n}) \\
&- (-1)^{mn} \big\{   \Theta_n (Q) ( [1]; \Theta_m (P) ([1]; u_0, \ldots, u_m), u_{m+1}, \ldots, u_{m+n} )\\
&+ \sum_{i=1}^n (-1)^{im} \Theta_n (Q) ([1]; u_0, \ldots, u_{i-1}, \Theta_m (P) ([1] + \cdots + [m+1]; u_i, \ldots, u_{i+m}), u_{i+m+1}, \ldots, u_{m+n})  \big\}\\
&= (-1)^{m+n+1} ~ u_0 (\llbracket P, Q \rrbracket (u_1, \ldots, u_{m+n})) = (\Theta_{m+n} \llbracket P, Q \rrbracket)([1]; u_0, u_1, \ldots, u_{m+n}).
\end{align*}
Similarly, for $u_1, \ldots, u_{m+n+1} \in M$,
\begin{align*}
&\llceil \Theta_m (P), \Theta_n (Q) \rrceil ([m+n+1]; u_1, \ldots, u_{m+n+1}) \\
&= \big( \sum_{i=1}^{m+1} (-1)^{(i-1)n} \Theta_m (P) \circ_i \Theta_n (Q) ~-~ (-1)^{mn} \sum_{i=1}^{n+1} (-1)^{(i-1)m} \Theta_n (Q) \circ_i \Theta_m (P) \big) ([m+n+1]; u_1, \ldots, u_{m+n+1})\\
&= \sum_{i=1}^m (-1)^{(i-1)n} \Theta_m (P) ( [m+1] ; u_1, \ldots, u_{i-1}, \Theta_n (Q) ([1] + \cdots + [n+1]; u_i, \ldots, u_{i+n}), \ldots, u_{m+n+1})\\
&+ (-1)^{mn}  \Theta_m (P) ([m+1]; u_1, \ldots, u_m, \Theta_n (Q) ([n+1]; u_{m+1}, \ldots, u_{m+n+1})) \\
&- (-1)^{mn} \big\{    \sum_{i=1}^n (-1)^{(i-1)m} \Theta_n (Q) ( [n+1] ; u_1, \ldots, u_{i-1}, \Theta_m (P) ([1] + \cdots + [m+1]; u_i, \ldots, u_{i+m}), \ldots, u_{m+n+1})\\
&+ (-1)^{mn}  \Theta_n (Q) ([n+1]; u_1, \ldots, u_n, \Theta_m (P) ([m+1]; u_{n+1}, \ldots, u_{m+n+1})) \big\} \\
&= (\llbracket P, Q \rrbracket (u_1, \ldots, u_{m+n})) u_{m+n+1} = (\Theta_{m+n} \llbracket P, Q \rrbracket ) ([m+n+1]; u_1, \ldots, u_{m+n+1}).
\end{align*}
Finally, for $2 \leq r \leq m+n$, one can easily observe from the definition of the bracket $\llceil ~, ~ \rrceil$ and partial compositions (\ref{dend-partial-comp}) that 
\begin{align*}
\llceil \Theta_m (P), \Theta_n (Q) \rrceil ([r]; u_0, \ldots, u_{m+n}) = 0 =  (\Theta_{m+n} \llbracket P, Q \rrbracket ) ([r]; u_0, \ldots, u_{m+n}).
\end{align*}
Hence the desired result follows.
\end{proof}

Note that $\Theta_1 (T) = \pi_T$. Thus it follows from the above lemma that the following diagram commutes
\[
\xymatrix{
\mathrm{Hom}(M^{\otimes n }, A) \ar[r]^{d_H = (-1)^n \llbracket T, ~ \rrbracket } \ar[d]_{\Theta_n} &  \mathrm{Hom}(M^{\otimes n+1 }, A)  \ar[d]^{\Theta_{n+1}}\\
\mathrm{Hom}(\mathbb{K}[C_{n+1}] \otimes M^{\otimes n+1 }, M) \ar[r]_{\delta_{\pi_T}} &  \mathrm{Hom}(\mathbb{K}[C_{n+2}] \otimes M^{\otimes n+2 }, M).
}
\]

Hence we get the following.
\begin{prop}
Let $T : M \rightarrow A$ be an $\mathcal{O}$-operator on $A$ with respect to a bimodule $M$. 
Then the collection of maps $\{ \Theta_n \}$ induces a morphism $\Theta_\bullet : H^\bullet (M, A) \rightarrow H^{\bullet +1}_{\mathrm{dend}} (M, M)$ from the cohomology of the $\mathcal{O}$-operator $T$ to the cohomology of the induced dendriform algebra structure on $M$.
\end{prop}

\subsection{Relation with the cohomology of $\mathcal{O}$-operator on Lie algebra}

Let $(\mathfrak{g}, [~,~])$ be a Lie algebra and $\varrho : \mathfrak{g} \rightarrow \mathfrak{gl}(M)$ be a representation of $\mathfrak{g}$ on a vector space $M$. An $\mathcal{O}$-operator on $\mathfrak{g}$ with respect to the representation $M$ is a linear map $T : M \rightarrow \mathfrak{g}$ satisfying
\begin{align*}
[T(m), T(n)] = T \big( \varrho (Tm)(n) - \varrho (Tn)(m)  \big), ~~~ \text{ for } m, n \in M.
\end{align*}
A Rota-Baxter operator of weight $0$ on a Lie algebra $\mathfrak{g}$ is an $\mathcal{O}$-operator of $\mathfrak{g}$ with respect to the adjoint representation on $\mathfrak{g}$. Thus, an $\mathcal{O}$-operator is a generalization of Rota-Baxter operator \cite{guo-book}.

Let $(A, \cdot )$ be an associative algebra and $M$ an $A$-bimodule. Consider the Lie algebra structure on $A$ with the commutator bracket $[a, b]_C := a \cdot b - b \cdot a.$ Then $M$ can be given a Lie algebra representation via $\varrho : A \rightarrow \mathfrak{gl}(M), ~ \varrho (a) (m) = am -ma$, for $a \in A, ~ m \in M$. We denote this representation by $(M, \varrho).$

With the above notations, we have the following.
\begin{prop}
The collection of maps $S_n : \mathrm{Hom}(A^{\otimes n}, M) \rightarrow \mathrm{Hom}(\wedge^n A, M)$ defined by
\begin{align*}
S_n(f) (a_1, \ldots, a_n) = \sum_{\sigma \in \mathbb{S}_n} (-1)^\sigma f (a_{\sigma(1)}, \ldots, a_{\sigma(n)})
\end{align*}
is a morphism from the Hochschild cochain complex of $A$ with coefficients in the bimodule $M$ to the Chevalley-Eilenberg complex of the commutator Lie algebra $A$ with coefficients in the representation $M$.
\end{prop}

This is standard and can be proved in a various way. First, it can be checked that $\{S_n \}$ maps Gerstenhaber bracket to the Nijenhuis-Richardson bracket. Note that an associative structure on $A$ and bimodule $M$ can be described by a Maurer-Cartan element in Gerstenhaber Lie bracket (cf. Proposition \ref{mlr}) while a Lie algebra $\mathfrak{g}$ and a representation on $M$ can be described by a Maurer-Cartan element in the Nijenhuis-Richardson bracket. The result follows as the differentials in both cases are induced by respective Maurer-Cartan elements and $\{ S_n \} $ maps the Maurer-Cartan element for the associative structure to the corresponding Maurer-Cartan element for Lie algebra structure.

\begin{prop}\label{o-o-ass-lie}
Let $T : M \rightarrow A$ be an $\mathcal{O}$-operator on an associative algebra $(A, \cdot)$ with respect to an $A$-bimodule $M$. Then $T$ is also an $\mathcal{O}$-operator on the commutator Lie algebra $(A, [~,~]_C)$ with respect to the representation $(M, \varrho)$.
\end{prop}

\begin{proof}
For any $m, n \in M$, we have
\begin{align*}
[T(m), T(n) ] =~& T(m) \cdot T(n) - T(n) \cdot T(m) \\
=~& T(m T(n) T(m) n) - T (n T(m) + T(n) m ) \\
=~& T (T(m) n - n T(m)) - T (T(n)m - m T(n)) \\
=~& T (\varrho (Tm) (n) - \varrho (Tn)(m)).
\end{align*}
Hence the proof.
\end{proof}

Let $T$ be an $\mathcal{O}$-operator on the associative algebra $A$ with respect to the bimodule $M$. Then $M$ carries an associative structure given by (\ref{m-ass}) and there is an $M$-bimodule structure on $A$ given by Lemma \ref{new-rep-o}. The Hochschild cohomology of $M$ with coefficients in the bimodule $A$ is by definition the cohomology of the $\mathcal{O}$-operator $T$. 

Note that the commutator Lie bracket on $M$ is given by 
\begin{align*}
[m, n ]_C = m \star n - n \star m =  m T(n) + T(m) n - n T(m) - T(n)m
\end{align*}
and its Lie algebra representation on $A$ is given by
\begin{align*}
\varrho_A : M \rightarrow \mathfrak{gl}(A), ~~ \varrho_A (m) (a) = l_T(m, a) - r_T (a, m) = T(m) \cdot a - T(ma) - a \cdot T(m) + T( am ).
\end{align*}

On the other hand, $T$ induces an $\mathcal{O}$-operator on the commutator Lie algebra $(A, [~,~]_C)$ with respect to the representation $(M, \varrho)$ (Proposition \ref{o-o-ass-lie}). Hence by a result of \cite[Lemma 3.1]{tang} the vector space $M$ carries a Lie bracket 
\begin{align*}
[m, n] = \varrho (Tm)n - \varrho (Tn) m = T(m) n - n T(m) - T(n) m + m T(n)
\end{align*}
and a representation on $A$ given by $\varrho'_A (m) (a) = [T(m), a]_C + T (\varrho (a))m = T(m) \cdot a - a \cdot T(m) + T ( am -m a). $ Thus, the two Lie algebra structures on $M$ and two representations on $A$ are same.

In \cite{tang} the authors define the cohomology of an $\mathcal{O}$-operator on a Lie algebra to be the Chevalley-Eilenberg cohomology of $M$ with coefficients in $A$. Thus, we obtain the following.

\begin{prop}\label{ass-lie-mor}
The collection of maps $S_n : \mathrm{Hom}(M^{\otimes n}, A) \rightarrow \mathrm{Hom}(\wedge^n M, A)$ given by
\begin{align*}
S_n(f) (m_1, \ldots, m_n) = \sum_{\sigma \in \mathbb{S}_n} (-1)^\sigma f (m_{\sigma (1)}, \ldots, m_{\sigma (n)})
\end{align*}
induces a morphism from the cohomology of $T$ as an $\mathcal{O}$-operator on the associative algebra $A$ to the cohomology of $T$ as an $\mathcal{O}$-operator on the corresponding commutator Lie algebra.
\end{prop}

\medskip

An $\mathcal{O}$-operator on associative algebra also induce a pre-Lie algebra structure via its induced dendriform structure. Pre-Lie algebras first appeared in the work of Gerstenhaber while studying the algebraic structure of the Hochschild cohomology ring \cite{gers2}. The cohomology and deformation theory for such algebras has been studied in \cite{dz}.

A pre-Lie algebra is a vector space $P$ together with a bilinear composition $\circ$ on $P$ that satisfies the following identity
\begin{align*}
(a \circ b) \circ c - a \circ ( b \circ c) = (b \circ a) \circ c - b \circ ( a \circ c), ~~~ \text{ for all } a, b, c \in P.
\end{align*}
If $(D, \prec, \succ)$ is a dendriform algebra, then the new operation
\begin{align*}
a \circ b = a \succ b - b \prec a
\end{align*}
is a (left) pre-Lie product on $D$. 

It follows from Proposition \ref{o-dend} that if $T : M \rightarrow A$ is an $\mathcal{O}$-operator on $A$ with respect to the bimodule $M$, the vector space $M$ carries a pre-Lie structure with product
\begin{align}\label{o-ass-pre}
m \circ n = T(m) n - n T(m), ~~~~ \text{ for } m, n \in M.
\end{align}

An $\mathcal{O}$-operator $T: M \rightarrow \mathfrak{g}$ on a Lie algebra $\mathfrak{g}$ with respect to a representation $(M, \varrho)$ also induces a pre-Lie algebra structure on $M$ with product
\begin{align}\label{o-lie-pre}
m \circ n = \varrho ( Tm ) (n).
\end{align}
If $\mathfrak{g} = A_C$ is the commutator Lie algebra of an associative algebra $A$ and the $\mathfrak{g}$-representation on $M$ is induced from the $A$-bimodule structure on $M$, then the above two pre-Lie product on $M$ are same. This observation together with Proposition \ref{ass-lie-mor} and a result \cite[Theorem 3.6]{tang} enable us to find a morphism from the cohomology of an associative $\mathcal{O}$-operator to the cohomology of the pre-Lie algebra on $M$.

We will not explicitly recall the cohomology of pre-Lie algebras which one can find in \cite{dz}, also in \cite{tang}. Rather we only recall the $n$-th cochain group and refer \cite[Equation (14)]{tang} for the coboundary map. Let $(M, \circ)$ be a pre-Lie algebra. Its
 $n$-th cochain group is given by $C^n (M, M) = \text{ Hom} (\wedge^{n-1} M \otimes M , M)$. If $T : M \rightarrow \mathfrak{g}$ is an $\mathcal{O}$-operator on $\mathfrak{g}$ with respect to a representation $(M, \varrho)$, then the collection of maps
\begin{align*}
\Phi_n : \text{Hom} (\wedge^n M, \mathfrak{g}) &\rightarrow \text{Hom} (\wedge^n M  \otimes M, M )~~~~ \text { given by } \\
\Phi_n (f) (u_1, \ldots, u_{n}, u_{n+1}) &= \varrho (  f(u_1, \ldots, u_n )) (u_{n+1})
\end{align*}
induces a map $\Phi_* : H^*(M, \mathfrak{g}) \rightarrow H^{*+1} (M, M)$ from the cohomology of $T$ (as an $\mathcal{O}$-operator on the Lie algebra $\mathfrak{g}$) to the cohomology of the induced pre-Lie algebra (\ref{o-lie-pre}) on $M$.

Thus, by combining this fact with Proposition \ref{ass-lie-mor}, we obtain the following.
\begin{prop}
Let $T : M \rightarrow A$ be an $\mathcal{O}$-operator on $A$ with respect to a bimodule $M$. Then the collection of maps $\Psi_n : \mathrm{Hom} (M^{\otimes n}, A ) \rightarrow \mathrm{Hom} (\wedge^n M \otimes M, M)$ given by
\begin{align*}
(\Psi_n f) (u_1, \ldots, u_n, u_{n+1}) = \sum_{\sigma \in \mathbb{S}_n} (-1)^\sigma ~ f(u_{\sigma (1)}, \ldots, u_{\sigma (n)} )u_{n+1}  ~-~ \sum_{\sigma \in \mathbb{S}_n} (-1)^\sigma~  u_{n+1} f(u_{\sigma (1)}, \ldots, u_{\sigma (n)} )
\end{align*}
induces a map $\Psi_\bullet : H^\bullet (M, A) \rightarrow H^{\bullet +1}_{\mathrm{preLie}} (M, M)$ from the cohomology of the associative $\mathcal{O}$-operator $T$ to the cohomology of the induced pre-Lie algebra (\ref{o-ass-pre}) on $M$.
\end{prop}

\section{Deformations of $\mathcal{O}$-operators}\label{sec4}
In this section, we study the deformation problem of $\mathcal{O}$-operators following the classical approaches. We also describe the extension problem of a finite order deformation to the next order. 
\subsection{Linear deformations}

Let $T: M \rightarrow A$ be an $\mathcal{O}$-operator on an associative algebra $A$ with respect to the $A$-bimodule $M$. A (one-parameter) linear deformation of $T$ consists of a sum $T_t = T + t \mathfrak{T}$, for some $\mathfrak{T} \in \text{Hom} (M,A)$, such that $T_t$ is an $\mathcal{O}$-operator on $A$ with respect to the bimodule $M$, for all $t$. In such a case, we say that $\mathfrak{T}$ generates a linear deformation of $T$.

Therefore, if $\mathfrak{T}$ generates a linear deformation of $T$, then $T_t = T + t \mathfrak{T}$ must satisfy
\begin{align*}
T_t (u) \cdot T_t (v) = T_t (u T_t(v) + T_t (u) v), ~~~~ \text{ for all } t \text{ and } u, v \in M.
\end{align*}

This condition is equivalent to the followings (by equating coefficients of $t$ and $t^2$ from both side)
\begin{align}
T(u) \cdot \mathfrak{T} (v) + \mathfrak{T}(u) \cdot T(v) =~& T (u \mathfrak{T}(v) + \mathfrak{T}(u)v) + \mathfrak{T} (u T(v) + T(u) v), \label{cond-p}\\
\mathfrak{T}(u) \cdot \mathfrak{T}(v) =~& \mathfrak{T} ( u \mathfrak{T}(v) + \mathfrak{T}(u) v ). \label{cond-q}
\end{align}
Observe that the condition (\ref{cond-p}) is same as $d_H (\mathfrak{T}) = 0$. In other words, $\mathfrak{T}$ is a $1$-cocycle in the cohomology of the $\mathcal{O}$-operator $T$. The condition (\ref{cond-q}) says that $\mathfrak{T}$ is an $\mathcal{O}$-operator on the algebra $A$ with respect to the bimodule $M$.

In the following, we show that a linear deformation of an $\mathcal{O}$-operator induces linear deformation of the corresponding dendriform structure.

\begin{prop}
If $\mathfrak{T}$ generates a linear deformation of an $\mathcal{O}$-operator $T$ on $A$ with respect to the bimodule $M$, then
\begin{itemize}
\item[(i)] $m \prec_t n = m T(n) + t~ m \mathfrak{T}(n), ~ m \succ_t n = T(m) n + t~ \mathfrak{T}(m) n$ defines a linear deformation of the corresponding dendriform structure on $M$.
\item[(ii)] $m \star_t b = m \star n + t (m \mathfrak{T}(n) + \mathfrak{T}(m) n )$ defines a linear deformation of the corresponding associative structure on $M$.
\end{itemize}
\end{prop}

\begin{defn}
Two linear deformations $T^1_t = T + t \mathfrak{T}_1$ and $T^2_t = T + t \mathfrak{T}_2$ of an $\mathcal{O}$-operator $T$ are said to be equivalent if there exists an $a \in A$ such that $\big( \text{id}_A + t (\text{ad}^l_a - \text{ad}^r_a), \text{id}_M + t (l_a - r_a) \big)$ is a morphism of $\mathcal{O}$-operators from $T^1_t$ to $T^2_t$.
\end{defn}

It follows from the above definition that $ \text{id}_A + t (\text{ad}^l_a - \text{ad}^r_a)$ is an associative algebra morphism on $A$. In other words,
\begin{align*}
(\text{id}_A + t (\text{ad}^l_a - \text{ad}^r_a)) (b \cdot c) = (\text{id}_A + t (\text{ad}^l_a - \text{ad}^r_a))(b) \cdot (\text{id}_A + t (\text{ad}^l_a - \text{ad}^r_a))(c), ~~ \text{ for all } b, c \in A.
\end{align*}
This implies that
\begin{align}\label{comm-comm-zero}
(a \cdot b - b \cdot a) \cdot ( a \cdot c - c \cdot a) = 0.
\end{align}

From (\ref{eq1}) we have $(T + t \mathfrak{T}_2 ) \circ (\text{id}_M + t (l_a - r_a)) (m) = (\text{id}_A + t (\text{ad}^l_a - \text{ad}^r_a)) \circ (T + t \mathfrak{T}_1 ) (m).$ This implies that 
\begin{align}\label{1-cocycle-linear}
\mathfrak{T}_1 (m) - \mathfrak{T}_2 (m) =~& T ( am - ma ) - (a \cdot T(m) - T(m) \cdot a) \nonumber\\
=~& l_T (m, a) - r_T (a, m)
\end{align}
and 
\begin{align}
a \cdot \mathfrak{T}_1 (m) - \mathfrak{T}_1 (m) \cdot a = \mathfrak{T}_2 ( am -ma ).
\end{align}
Moreover, from (\ref{eq2}) we have $(a \cdot b - b \cdot a) (am -ma) = 0 $, or equivalently,
\begin{align}\label{ll-lr}
l_{(a \cdot b - b \cdot a)} \circ l_a = l_{(a \cdot b - b \cdot a)} \circ r_a, ~~~~ \text{ for } b \in A.
\end{align}
Similarly, the condition (\ref{eq3}) gives rise to
\begin{align}\label{rl-rr}
r_{(a \cdot b - b \cdot a)} \circ l_a = r_{(a \cdot b - b \cdot a)} \circ r_a, ~~~~ \text{ for } b \in A.
\end{align}

Note that from (\ref{1-cocycle-linear}) we have $\mathfrak{T}_1 (m) - \mathfrak{T}_2 (m) = d_H (a) (m).$ Thus we have the following.

\begin{thm}
Let  $T^1_t = T + t \mathfrak{T}_1$ and $T^2_t = T + t \mathfrak{T}_2$ be two equivalent linear deformations of an $\mathcal{O}$-operator $T$. Then $\mathfrak{T}_1$ and $\mathfrak{T}_2$ defines the same cohomology class in $H^1(M, A).$
\end{thm}

\begin{defn}
A linear deformation $T_t = T + t \mathfrak{T}$ of an $\mathcal{O}$-operator $T$ is said to be trivial if this deformation is equivalent to the deformation $T_t' = T.$
\end{defn}

In the following, we consider Nijenhuis elements associated with an $\mathcal{O}$-operator which induce trivial deformations. The following definition is motivated by the above discussions.

\begin{defn}\label{ass-nij-element}
An element $a \in A$ is called a Nijenhuis element associated to an $\mathcal{O}$-operator $T$ if $a$ satisfies
\begin{align*}
a \cdot (l_T (m, a) -  r_T (a, m)) ~-~ (l_T (m, a) - r_T (a, m) ) \cdot a = 0, \text{ for all } m \in M
\end{align*}
and the conditions (\ref{comm-comm-zero}), (\ref{ll-lr}) and (\ref{rl-rr}).
\end{defn}

We denote the set of all Nijenhuis elements associated to an $\mathcal{O}$-operator $T$ by $\mathrm{Nij}(T).$

\begin{remark}\label{nij-remark}
In \cite[Definition 4.8]{tang} the authors introduce Nijenhuis elements associated to an $\mathcal{O}$-operator on a Lie algebra to study their trivial deformations. More precisely, let $T : M \rightarrow \mathfrak{g}$ be an $\mathcal{O}$-operator on $(\mathfrak{g}, [~,~])$ with respect to a representation $(M, \varrho)$. An element $x \in \mathfrak{g}$ is a Nijenhuis element if it satisfies
\begin{align*}
[[x,y], [x, z]] = 0,  ~~ \quad \varrho ([x,y]) \varrho (x) = 0 \quad \text{ and } \quad [ x, [Tu, x] + T \varrho(x)(u) ] = 0,
\end{align*}
for all $y, z \in \mathfrak{g}$ and $u \in M$. It is easy to see from our Definition \ref{ass-nij-element} that if $a \in A$ is a Nijenhuis element for an associative $\mathcal{O}$-operator $T : M \rightarrow A$, then $a \in A$ is also a Nijenhuis element for the $\mathcal{O}$-operator $T$ on the commutator Lie algebra $(A, [~,~]_C).$
\end{remark}

Let $T_t = T + t \mathfrak{T}$ be a trivial deformation of an $\mathcal{O}$-operator $T$. Then there exists $a \in A$ such that $(\text{id}_A + t (\text{ad}^l_a - \text{ad}^r_a), \text{id}_M + t (l_a - r_a))$ is a morphism of $\mathcal{O}$-operators from $T_t$ to $T$. If follows from (\ref{comm-comm-zero})-(\ref{rl-rr}) with $\mathfrak{T}_1 = \mathfrak{T}$ and $\mathfrak{T}_2 = 0$, we get that $a \in A$ is a Nijenhuis element associated to $T$. Thus, a trivial deformation gives rise to a Nijenhuis element. The following result proves the converse.

\begin{thm}
Let $T : M \rightarrow A$ be an $\mathcal{O}$-operator on $A$ with respect to the bimodule $M$. For any Nijenhuis element $a \in \mathrm{Nij}(T)$, the sum $T_t = T + t \mathfrak{T}$ defines a trivial deformation of $T$, where $\mathfrak{T} := d_H (a).$
\end{thm}

\begin{proof}
To prove that $T_t$ is a deformation, we need to verify (\ref{cond-p}) and (\ref{cond-q}). The identity (\ref{cond-p}) holds trivially as $\mathfrak{T} = d_H (a)$ is a $1$-cocycle. The identity (\ref{cond-q}) is also straightforward and also follows from a long and tedious computation. We refer \cite[Theorem 4.9]{tang} for a similar computation.

Finally, since $a$ is a Nijenhuis element, it follows that the pair $(\text{id}_A + t (\text{ad}^l_a - \text{ad}^r_a), \text{id}_M + t (l_a - r_a))$ gives a morphism of $\mathcal{O}$-operators from $T_t$ to $T$. Hence the result follows.
\end{proof}

In \cite{tang} the authors gave another motivation of Nijenhuis elements associated with an $\mathcal{O}$-operator on a Lie algebra. More precisely, they relate them with Nijenhuis operator on pre-Lie algebras.
Recall that a linear map $N :M \rightarrow M$ on a pre-Lie algebra $(M, \circ)$ is called a Nijenhuis operator \cite{wang} if it satisfies
\begin{align*}
N(m) \circ N(n) = N ( N (m) \circ n + m \circ N (n) - N (m \circ n)), \text{ for all } m, n \in M.
\end{align*}

In view of Remark \ref{nij-remark} and \cite[Proposition 4.11]{tang}, we obtain the following.
\begin{prop}
Let $a \in \mathrm{Nij}(T)$ be a Nijenhuis element associated to an associative $\mathcal{O}$-operator $T : M \rightarrow A$. Then $l_a - r_a$ is a Nijenhuis operator for the pre-Lie structure (\ref{o-ass-pre}) on $M$. Hence it is also a Nijenhuis operator for the Lie algebra structure on $M$.
\end{prop}

\medskip

\subsection{Formal deformations}\label{subsec-formal}

In this subsection, we study formal one-parameter deformation of $\mathcal{O}$-operators following the approach of Gerstenhaber \cite{gers}. 

Let $A$ be an associative algebra. Consider the space $A [[ t]]$ of formal power series in $t$ with coefficients in $A$. Then $A [[ t]]$ is an associative algebra over $\mathbb{K}[[t]]$. Note that, when $A$ is finite dimensional, we have $A [[t]] \cong A \otimes_{\mathbb{K}} \mathbb{K}[[t]]$. Moreover, if $M$ is an $A$-bimodule, then $M[[t]]$ can be given the structure of an $A [[t]]$-bimodule with the obvious left and right actions.

Let $T: M \rightarrow A$ be an $\mathcal{O}$-operator on the associative algebra $A$ with respect to the bimodule $M$.

\begin{defn}
A formal one-parameter deformation of $T$ consists of a formal sum
\begin{align*}
T_t = T_0 + t T_1 + t^2 T_2 + \cdots \in \text{Hom}(M, A)[[t]]
\end{align*}
with $T_0 = T$ such that $T_t : M [[t]] \rightarrow A[[t]]$ is an $\mathcal{O}$-operator on $A[[t]]$ with respect to the bimodule $M[[t]]$, i.e.
\begin{align}\label{formal-def}
T_t(m) \cdot T_t (n) = T_t ( m T_t (n) + T_t (m) n ), ~~~ \text{ for all } m, n \in M.
\end{align}
\end{defn}

Note that the condition (\ref{formal-def}) is equivalent to a system of equations: for each $k \geq 0$,
\begin{align*}
\sum_{i+ j = k} T_i (m) \cdot T_j (n) = \sum_{i+j = k} T_i ( m T_j (n) + T_j (m) n ).
\end{align*}
The above identity holds for $k = 0$ as $T$ is an $\mathcal{O}$-operator. For $k = 1$, it amounts that
\begin{align*}
T(m) \cdot T_1 (n) + T_1 (m) \cdot T(n) = T ( m T_1 (n) + T_1 (m) n ) + T_1 (m T(n) + T(m)n), ~~ \text{ for } m, n \in M.
\end{align*}
Hence it follows from (\ref{1-coc}) that $d_H (T_1) = 0$. Thus, we have the following.

\begin{prop}\label{lin-term-1-co}
Let $T_t = \sum_{i \geq 0} t^i T_i$ be a formal deformation of an $\mathcal{O}$-operator $T$ on the algebra $A$ with respect to the bimodule $M$. Then the linear term $T_1 \in \mathrm{Hom} (M, A)$ is a $1$-cocycle on the cohomology of the $\mathcal{O}$-operator $T$.
\end{prop}

The $1$-cocycle $T_1 \in Z^1(M, A)$ is called the infinitesimal of the formal deformation $T_t$.

In \cite{das} the author study formal deformations of dendriform algebras and their relation with the cohomology described in Subsection \ref{subsec-dend-coho}. Here we relate deformations of an $\mathcal{O}$-operator with deformations of the corresponding dendriform algebra. 

\begin{prop}
Let $T_t = \sum_{i \geq 0 } t^i T_i$ be a formal deformation of an $\mathcal{O}$-operator $T$. Then the formal sums
\begin{align*}
m \prec_t n = \sum_{i \geq 0} t^i m T_i (n) \quad \text{ and } \quad
m \succ_t n = \sum_{i \geq 0} t^i T_i (m) n
\end{align*}
defines a formal deformation of the dendriform structure on $M$.
\end{prop}

\begin{corollary}
It follows that the formal sum $m \star_t n = \sum_{i \geq 0} t^i ( m T_i (n) + T_i (m) n)$ defines a formal deformation of the associative structure on $M$.
\end{corollary}

\begin{defn} Two formal deformations $T_t = \sum_{i \geq 0} t^i T_i$ and $\overline{T}_t = \sum_{i \geq 0} t^i \overline{T}_i$ of an $\mathcal{O}$-operator $T$ are said to be equivalent if there exists an element $a \in A$, linear maps $\phi_i \in \text{Hom} ( A, A)$ and $\psi_i \in \text{Hom} (M, M)$, for $i \geq 2$, such that
\begin{align*}
\phi_t = \text{id}_A + t (\text{ad}^l_a - \text{ad}^r_a) + \sum_{i \geq 2} t^i \phi_i \quad \text{ and } \quad
\psi_t = \text{id}_M + t (l_a - r_a) +  \sum_{i \geq 2} t^i \psi_i
\end{align*}
defines a morphism of $\mathcal{O}$-operators from $T_t$ to $\overline{T}_t$.
\end{defn}

Therefore, from Definition \ref{o-op-map} the following identities must hold
\begin{itemize}
\item[(i)] $\phi_t (b) \cdot \phi_t (c) = \phi_t ( b \cdot c)$, for all $b , c \in A$,
\item[(ii)] $\overline{T}_t \circ \psi_t = \phi_t \circ T_t$,
\item[(iii)] $\phi_t (b) \psi_t (m) = \psi_t (bm)$,
\item[(iv)] $\psi_t (m) \phi_t (b) = \psi_t (mb)$, for $b \in A, m \in M$.
\end{itemize}

It follows from (ii) that $\overline{T}_t \circ \psi_t (m) = \phi_t \circ T_t (m)$. By equating coefficients of $t$ from both side, we obtain
\begin{align*}
T_1 (m) - \overline{T}_1 (m) = T(m) \cdot a - T(ma) - a \cdot T(m) + T (am) = d_H (a) (m).
\end{align*}

Thus, we have the following.

\begin{prop}
The infinitesimals of equivalent deformations are cohomologous, i.e. they lie on the same cohomology class.
\end{prop}

\begin{defn}An $\mathcal{O}$-operator $T$ is said to be rigid if every deformation $T_t$ of $T$ is equivalent to the deformation $\overline{T}_t = T$. 
\end{defn}
The next proposition describes a cohomological obstruction for the rigidity of an $\mathcal{O}$-operator.

\begin{prop}
Let $T$ be an $\mathcal{O}$-operator on an associative algebra $A$ with respect to a bimodule $M$. If $Z^1 (M, A) = d_H (\mathrm{Nij}(T))$ then $T$ is rigid.
\end{prop}

\begin{proof}
Let $T_t = \sum_{i \geq 0} t^i T_i$ be any formal deformation of the $\mathcal{O}$-operator $T$. Then by Proposition \ref{lin-term-1-co}, the linear term $T_1$ is in $Z^1(M, A)$. Therefore, by assumption, we have $T_1 = d_H (a)$, for some $a \in \text{Nij} (T)$. Set
\begin{align*}
\phi_t := \text{id}_A + t (\text{ad}^l_a - \text{ad}^r_a )  \quad \text{ and }  \quad \psi_t := \text{id}_M + t (l_a - r_a)
\end{align*}
and define $\overline{T}_t = \phi_t \circ T_t \circ \psi_t^{-1}$. Then $T_t$ is equivalent to $\overline{T}_t$. Moreover, from the definition of $\overline{T}_t$, we have
\begin{align*}
\overline{T}_t (m) = (\text{id}_A + t (\text{ad}^l_a - \text{ad}^r_a )) \circ (\sum_{i \geq 0} t^i T_i ) \circ (\text{id}_M - t (l_a - r_a) + t^2 (l_a - r_a)^2 - \cdots ) (m), ~~~ \text{ for } m \in M.
\end{align*}
Hence
\begin{align*}
\overline{T}_t (m) \quad (\text{mod } t^2) =~& (\text{id}_A + t (\text{ad}^l_a - \text{ad}^r_a )) \circ (T + t T_1) (m - t(am-ma))  \quad (\text{mod } t^2)\\
=~& (\text{id}_A + t (\text{ad}^l_a - \text{ad}^r_a )) (T(m) + t T_1 (m) - t T(am -ma))  \quad (\text{mod } t^2)\\
=~& T(m) + t \big( a \cdot T(m) - T(m) \cdot a -  T(am - ma) +  T_1 (m) \big).
\end{align*}
The coefficient of $t$ is zero as $T_1 = d_H (a)$. See (\ref{zero-diff}) for instance. Therefore, $\overline{T}_t$ is of the form $\overline{T}_t = T + \sum_{i \geq 2} t^i \overline{T}_i$. By repeating this argument, one get the equivalence between $T_t$ and $T$. Hence the proof.
\end{proof}

\medskip

\subsection{Extensions of finite order deformation}

In the following, we consider the problem of extension of a finite order deformation of an $\mathcal{O}$-operator. 

Let $(A, \cdot )$ be an associative algebra and $M$ be an $A$-bimodule. Consider the space $A[[t]] / (t^{n+1})$ which inherits an associative algebra structure over $\mathbb{K}[[t]]/ (t^{n+1})$. Moreover, $M[[t]]/ (t^{n+1})$ is an $A[[t]] / (t^{n+1})$-bimodule. Let $T : M \rightarrow A$ be an $\mathcal{O}$-operator on $A$ with respect to the $A$-bimodule $M$. An order $n$ deformation of $T$ consists of a sum $T_t = \sum_{i= 0}^n t^i T_i$ with $T_0  = T$ such that $T_t$ defines an $\mathcal{O}$-operator on the algebra $A[[t]]/ (t^{n+1})$ with respect to the bimodule $M[[t]] / (t^{n+1})$. In other words, one must have
\begin{align*}
T_t (m) \cdot T_t (n) = T_t ( m T_t (n) + T_t (m) n )  ~~ \quad \text{mod} (t^{n+1}).
\end{align*}
If there exists a linear map $T_{n+1} : M \rightarrow A$ such that $\widetilde{T}_t = T_t + t^{n+1} T_{n+1} = \sum_{i=0}^{n+1} t^i T_i$ defines a deformation of order $n+1$, then we say that $T_t$ extends to a deformation of next order.

Since we assume that $T_t = \sum_{i=0}^n t^i T_i$ is a deformation of order $n$, the following deformation equations must hold:
\begin{align*}
\sum_{i+j = k} T_i (m) \cdot T_j (n) = \sum_{i+j = k} T_i (m T_j (n) + T_j (m) n), ~~~ \text{ for } k = 0, 1, \ldots, n,
\end{align*}
or equivalently,
\begin{align}\label{def-e}
\llbracket T, T_k \rrbracket = - \frac{1}{2} \sum_{i+j = k, i, j \geq 1} \llbracket T_i, T_j \rrbracket,  ~~~ \text{ for } k = 0, 1, \ldots, n.
\end{align}

If it is extendable, then one more deformation equation need to be satisfied, namely,
\begin{align*}
\sum_{i+j = n+1} T_i (m) \cdot T_j (n) = \sum_{i+j = n+1} T_i (m T_j (n) + T_j (m) n).
\end{align*}
This is equivalent to
\begin{align*}
 & T(m) \cdot  T_{n+1} (n) ~+~ T_{n+1} (m) \cdot T(n) ~+ \sum_{i+j=n+1, i, j \geq 1} T_i(m) \cdot T_j (n) \\ &~=  T ( m T_{n+1} (n) + T_{n+1} (m)n) ~+~ T_{n+1} (m T(n) + T(m) n) ~+ \sum_{i+j = n+1, i, j \geq 1} T_i (m T_j (n) + T_j(m)n ),
\end{align*}
or equivalently,
\begin{align*}
T ( m T_{n+1} (n) + T_{n+1} (m)n) ~+~ T_{n+1} (m T(n) + T(m) n) - T(m) \cdot  T_{n+1} (n) ~-~ T_{n+1} (m) \cdot T(n)  \\
=  \sum_{i+j = n+1, i, j \geq 1} \bigg[ T_i (m) \cdot T_j (n) - T_i (m T_j (n) + T_j(m)n ) \bigg].
\end{align*}
This is same as
\begin{align}\label{obstr}
(d_H (T_{n+1})) (m, n) = \sum_{i+j = n+1, i, j \geq 1} \bigg[ T_i (m) \cdot T_j (n) - T_i (m T_j (n) + T_j(m)n ) \bigg].
\end{align}
Observe that the right hand side of the above equation doesn't involve $T_{n+1}$. It is called the obstruction to extend the deformation $T_t$.

\begin{prop}\label{obs-2-co}
The map $\mathrm{Ob}_T : M^{\otimes 2} \rightarrow A$ defined by
\begin{align*}
\mathrm{Ob}_T (m, n) = \sum_{i+j = n+1, i, j \geq 1} \bigg[ T_i (m) \cdot T_j (n) - T_i (m T_j(n) + T_j(m)n ) \bigg], ~~ \text{ for } m, n \in M,
\end{align*}
is a $2$-cocycle in the cohomology complex of $T$. In other words, $d_H (\mathrm{Ob}_T) = 0.$
\end{prop}

\begin{proof}
Note that from (\ref{t-t}) we have
\begin{align*}
\mathrm{Ob}_T (m,n) = - \frac{1}{2} \sum_{i+j = n+1, i, j \geq 1} \llbracket T_i, T_j \rrbracket (m,n).
\end{align*}
Therefore,
\begin{align*}
d_H (\mathrm{Ob}_T ) =~& (-1)^2 \llbracket T, \mathrm{Ob}_T \rrbracket \\
=~& - \frac{1}{2} \sum_{i+j = n+1, i, j \geq 1} \llbracket T, \llbracket T_i, T_j \rrbracket \rrbracket \\
=~& - \frac{1}{2} \sum_{i+j = n+1, i, j \geq 1}  \bigg(  \llbracket \llbracket T,T_i \rrbracket, T_j \rrbracket - \llbracket T_i, \llbracket T, T_j \rrbracket \rrbracket  \bigg) \\
=~& \frac{1}{4} \sum_{i_1 + i_2 + j = n+1 , i_1, i_2 , j \geq 1}  \llbracket \llbracket T_{i_1}, T_{i_2} \rrbracket, T_j \rrbracket - \frac{1}{4} \sum_{i+ j_1 + j_2 = n+1, i, j_1, j_2 \geq 1} \llbracket T_i, \llbracket T_{j_1}, T_{j_2} \rrbracket \rrbracket \quad (\text{by } (\ref{def-e}))\\
=~& \frac{1}{2} \sum_{i+j+k = n+1, i, j, k \geq 1} \llbracket \llbracket T_i, T_j \rrbracket, T_k \rrbracket = 0.
\end{align*}
\end{proof}

This shows that the obstruction gives rise to a cohomology class $[\mathrm{Ob}_T] \in H^2 (M, A).$ This is called the obstruction class to extend the deformation.

As a consequence of Equation (\ref{obstr}) and Proposition \ref{obs-2-co} we obtain the following.

\begin{thm}
An order $n$ deformation $T_t$ extends to a deformation of next order if and only if the obstruction class $[\mathrm{Ob}_T]$ is trivial.
\end{thm}

\begin{corollary}
If $H^2 (M, A) = 0$ then every finite order deformation of $T$ extends to a deformation of next order.
\end{corollary}

\begin{corollary}
If $H^2 (M, A) = 0$ then every $1$-cocycle in $Z^1 (M, A)$ is the infinitesimal of some formal deformation of the $\mathcal{O}$-operator $T$.
\end{corollary}

\medskip

\section{Applications}\label{sec5}

It was shown in Introduction that $\mathcal{O}$-operators generalize associative Rota-Baxter operators of weight $0$.  They also generalize associative {\bf r}-matrices \cite{aguiar2}. Therefore, we may study deformations of weight zero Rota-Baxter operators and associative {\bf r}-matrices as a particular case of deformation of $\mathcal{O}$-operators as described in the previous section.

\subsection{Rota-Baxter operators}
It has been shown that a Rota-Baxter operator (of weight $0$) on an associative algebra $A$ can be seen as an $\mathcal{O}$-operator on $A$ with respect to the adjoint bimodule $A$. Therefore, by considering the adjoint bimodule in place of arbitrary bimodule, we get similar results of Sections \ref{sec2}-\ref{sec4}. Here we collect a few of them.

By combining Theorem \ref{new-gla} and Theorem \ref{new-gla-2} in the case of adjoint bimodule, we have the following.

\begin{thm}
Let $A$ be an associative algebra. Then
\begin{itemize}
\item[(i)] there exists a graded Lie algebra structure $\llbracket ~ , ~ \rrbracket$ defined by (\ref{derived-bracket}) on the graded vector space $C^\bullet (A, A) = \oplus_{n \geq 0} C^n (A, A).$

\item[(ii)] A linear map $R : A \rightarrow A$ is a Rota-Baxter operator of weight $0$ if and only if $R$ is a Maurer-Cartan element in the graded Lie algebra $(C^\bullet (A, A), \llbracket ~,~\rrbracket)$. Consequently, a Rota-Baxter operator $R$ of weight $0$ induces a differential $d_R = \llbracket R, ~ \rrbracket$ on the gLa $(C^\bullet (A, A), \llbracket ~,~\rrbracket )$ to make it into a dgLa.

\item[(iii)] Moreover, for a linear map $R' : A \rightarrow A$, the sum $R + R'$ is also a Rota-Baxter operator of weight $0$ if and only if $R'$ satisfies
\begin{align*}
d_R (R') + \frac{1}{2} \llbracket R', R' \rrbracket = 0.
\end{align*}
\end{itemize}
\end{thm}

Given a Rota-Baxter operator (of weight $0$) on the algebra $A$, the vector space $A$ carries a dendriform structure \cite{aguiar}. Hence, $A$ carries a new associative product $a \star b = a \cdot R(b) + R(a) \cdot b$, for $a, b \in A$. This associative algebra $(A, \ast )$ has a bimodule representation on $A$ given by
\begin{align}\label{rota-bi-bi}
l_{a} (b) = R(a) \cdot b - R (a \cdot b ) \quad \text{ and } \quad r_a (b) = b \cdot R(a) - R ( b \cdot a).
\end{align}
The cohomology of the associative algebra $(A, \star)$ with coefficients in the above bimodule structure on $A$ is called the cohomology of the Rota-Baxter operator $R$.

\begin{remark}
Note that the associative algebra $(A, \star)$ has two more bimodule structure on $A$. The first one is given by the adjoint bimodule $\text{ad}^l_a ( b ) = a \star b$ and $\text{ad}^r_a (b) = b \star a$. The second one is given by $l'_a (b ) = R(a) \cdot b$ and $r'_a (b) = b \cdot R(a)$. However, neither of these two bimodule structures are same (in general) with that of (\ref{rota-bi-bi}).
\end{remark}

Next, we consider deformations of Rota-Baxter operators of weight $0$. Let $R$ be a Rota-Baxter operator of weight $0$ on an associative algebra $A$.

\begin{defn}
\begin{itemize}
\item[(i)] A linear map $\mathcal{R}: A \rightarrow A$ is said to generate a linear deformation of $R$ if for each $t \in \mathbb{K}$, the sum $R_t = R + t \mathcal{R}$ is a Rota-Baxter of weight $0$ on $A$.
\item[(ii)] Two linear deformations $R^1_t = R + t \mathcal{R}_1$ and $R^2_t = R + t \mathcal{R}_2$ of a Rota-Baxter operator $R$ of weight $0$ are equivalent if there exists $a \in A$ such that $(\text{id}_A + t (\text{ad}^l_a - \text{ad}^r_a), ~ \text{id}_A + t (\text{ad}^l_a - \text{ad}^r_a ))$ is a morphism of $\mathcal{O}$-operators from $R^1_t$ to $R^2_t$.
\item[(iii)] A linear deformation $R_t = R + t \mathcal{R}$ is said to be a trivial deformation if it is equivalent to the deformation $R^2_t = R$.
\end{itemize}
\end{defn}

\begin{prop}
If $\mathcal{R} : A \rightarrow A$ generates a linear deformation of a Rota-Baxter operator $R$ of weight $0$, then $\mathcal{R}$ is a $1$-cocycle in the cohomology of $R$. Moreover, if two linear deformations generated by $\mathcal{R}_1$ and $\mathcal{R}_2$ are equivalent, then $\mathcal{R}_1$ and $\mathcal{R}_2$ are cohomologous.
\end{prop}

\begin{defn}
An element $a \in A$ is a Nijenhuis element associated to a Rota-Baxter operator $R$ of weight $0$ if $a$ satisfies
\begin{align*}
&( a \cdot b - b \cdot a ) \cdot ( a \cdot c - c \cdot a ) = 0,\\
&[ a, R(b) \cdot a - a \cdot R(b) - R( a \cdot b - b \cdot a) ]_C = 0, ~~~~ \text{ for all } b \in A. 
\end{align*}
\end{defn}
The set of Nijenhuis elements of a Rota-Baxter operator $R$ of weight $0$ is denoted by $\mathrm{Nij}(R).$

\begin{prop}
Any trivial deformation of a Rota-Baxter operator $R$ induces a Nijenhuis element in $\mathrm{Nij}(R)$. Moreover,
for any Nijenhuis element $a \in \mathrm{Nij}(R)$, the linear map $\mathcal{R} = d_H (a)$ generates a trivial deformation of the Rota-Baxter operator $R$.
\end{prop}

\medskip

One may also formulate the  formal deformation $R_t = \sum_{i \geq 0} t^i R_i$  of a Rota-Baxter operator $R$ of weight $0$. This is similar to Subsection \ref{subsec-formal}.  The linear term $R_1$ turns out to be a $1$-cocycle in the cohomology of the Rota-Baxter operator $R$, called the infinitesimal of the deformation. Moreover, equivalent deformations have cohomologous infinitesimals. Finally, the vanishing of the second cohomology of $R$ allows one to extend a finite order deformation of $R$ to deformation of next order. 

\subsection{Associative {\bf r}-matrices}

The notion of associative {\bf r}-matrix was first introduced by Aguiar as an associative analogue of classical $r$-matrix \cite{aguiar2}. An associative {\bf r}-matrix can be seen as an $\mathcal{O}$-operator. 

Let $(A, \cdot)$ be an associative algebra and $r \in \wedge^2 A$. Note that $r$ induces a skew-symmetric linear map $r^\sharp : A^\ast \rightarrow A$ by
\begin{align*}
\langle \beta, r^\sharp (\alpha ) \rangle = r ( \alpha, \beta), ~~ \text{ for } \alpha, \beta \in A^*,
\end{align*}
where $\langle ~, ~ \rangle$ denotes the pairing between the elements of $A$ and $A^*$.

\begin{defn} Let $(A, \cdot)$  be an associative algebra.
An element  $r \in \wedge^2 A$ is called an associative {\bf r}-matrix if r satisfies $[[r, r ]] = 0$, where $[[r, r]] \in A \otimes A \otimes A$ is given by
\begin{align*}
[[r,r]] ( \alpha, \beta, \gamma) = \langle r^\sharp (\alpha) \cdot r^\sharp (\beta), \gamma \rangle + \langle r^\sharp (\beta) \cdot r^\sharp (\gamma), \alpha \rangle + \langle r^\sharp (\gamma) \cdot r^\sharp (\alpha), \beta \rangle,
\end{align*}
for $\alpha, \beta, \gamma \in A^*.$
\end{defn}

The relation between associative r-matrix and $\mathcal{O}$-operator is given by the following \cite{bai1}.

\begin{prop}\label{r-o-char}
An element $r \in \wedge^2 A$ is an associative {\bf r}-matrix if and only if the induced map $r^\sharp : A^* \rightarrow A$ is an $\mathcal{O}$-operator on $A$ with respect to the coadjoint $A$-bimodule $A^*$.
\end{prop}

Thus, if $r$ is an associative {\bf r}-matrix, it follows from Proposition \ref{o-dend} that $A^*$ carries an associative structure given by
\begin{align}
\alpha \star \beta = \text{ad}^{\ast l}_{r^\sharp (\alpha)} \beta + \text{ad}^{\ast r}_{r^\sharp (\beta)} \alpha , ~~~~ \text{ for } ~~~ \alpha, \beta \in A^*.
\end{align}
(Don't confuse with two $r$ in the last term $\text{ad}^{\ast r}_{r^\sharp (\beta)} \alpha$. The top $r$ always indicate the right (coadjoint) action whereas the lower one is the {\bf r}-matrix.) Moreover, it follows from Lemma \ref{new-rep-o} that the vector space $A$ carries an $A^*$-bimodule structure with left and right actions
\begin{align*}
l_{\alpha} (a) =~& r^\sharp (\alpha) \cdot a - r^\sharp (\mathrm{ad}^{\ast r}_a (\alpha)),\\
r_{\alpha} (a) =~& a \cdot r^\sharp (\alpha) - r^\sharp (\mathrm{ad}^{\ast l}_a (\alpha)),~~~ \text{ for } \alpha \in A^*, a \in A.
\end{align*}

\begin{remark}
The above bimodule representation of $A^*$ on the vector space $A$ can be seen as the coadjoint representation of $A^*$ on the dual space $A$. To see this, we observe that for any $\beta \in A^*,$
\begin{align*}
\langle \text{ad}^{* l}_\alpha a, \beta \rangle = \langle \beta \star \alpha, a \rangle 
=~& \langle   \text{ad}^{\ast l}_{r^\sharp (\beta)} \alpha + \text{ad}^{\ast r}_{r^\sharp (\alpha)} \beta    , a \rangle \\
=~& \alpha ( a \cdot r^\sharp (\beta)) + \beta (r^\sharp (\alpha) \cdot a) \\
=~& \langle \text{ad}^{\ast r}_a (\alpha), r^\sharp \beta \rangle + \langle r^\sharp (\alpha) \cdot a , \beta \rangle \\
=~& r (\beta,  \text{ad}^{\ast r}_a (\alpha)) + \langle r^\sharp (\alpha) \cdot a , \beta \rangle \\
=~& \langle - r^\sharp (\text{ad}^{\ast r}_a (\alpha)) + r^\sharp (\alpha) \cdot a , \beta \rangle = \langle l_\alpha (a), \beta \rangle.
\end{align*}
Hence we have $\text{ad}^{* l}_\alpha a = l_\alpha (a).$ Similarly, one can prove that $\text{ad}^{* r}_\alpha (a) = r_\alpha (a)$.
\end{remark}

Note that the cohomology of the $\mathcal{O}$-operator $r^\sharp : A^* \rightarrow A$ associated to an {\bf r}-matrix $r \in \wedge^2 A$ is given by the Hochschild cohomology of $(A^*, \star)$ with respect to the above representation on $A$.

\medskip

Next, we introduce a notion between two associative {\bf r}-matrices which induces a morphism between corresponding $\mathcal{O}$-operators. This is motivated from \cite{tang} for Lie algebra case. This notion will help us to define an equivalence between two deformations of an {\bf r}-matrix. Let $(A, \cdot)$ be an associative algebra and $r_1, r_2 \in \wedge^2 A$ be two associative {\bf r}-matrices.

\begin{defn}
\begin{itemize}
\item[(i)] A weak morphism from $r_1$ to $r_2$ consists of a pair $(\phi, \psi)$ of an associative algebra morphism $\phi :A  \rightarrow A$ and a linear map $\psi :  A \rightarrow A$ satisfying
\begin{align*}
(\psi \otimes \text{id}_A ) (r_2 ) =~& (\text{id}_A \otimes \phi ) (r_1),\\
\psi ( \phi (a) \cdot b ) =~& a \cdot \psi (b),\\
\psi ( a \cdot \phi (b) ) =~& \psi (a) \cdot b.
\end{align*}
\item[(ii)] A weak morphism $(\phi, \psi)$ is called a weak isomorphism if $\phi$ and $\psi$ are linear isomorphisms.
\end{itemize}
\end{defn}

The following result describes the relation between weak (iso)morphism of associative {\bf r}-matrices and (iso)morphism between corresponding $\mathcal{O}$-operators.

\begin{prop} Let $r_1, r_2 \in \wedge^2 A$ be two associative {\bf r}-matrices on an associative algebra $A$. Then a pair $(\phi, \psi)$ is a weak (iso)morphism from $r_1$ to $r_2$
if and only if $(\phi, \psi^*)$ is a  (iso)morphism of $\mathcal{O}$-operators from $r_1^\sharp$ to $r_2^\sharp$.
\end{prop}

\begin{proof}
Let $r_1 = \sum_{i} a_i \otimes b_i$ and $r_2 = \sum_j x_j \otimes y_j$. Suppose that the pair $(\phi, \phi^*) $ defines a morphism of $\mathcal{O}$-operators from $r_1^\sharp$ to $r_2^\sharp$. Then by Definition \ref{o-op-map} we have
\begin{align*}
r_2^\sharp \circ \psi^* = \phi \circ r_1^\sharp,~~~~ \quad \phi (a) \psi^*(\xi) = \psi^* ( a \xi ) ~~~~ \text{and } ~~~~ \psi^* (\xi) \phi (a) = \psi^* (\xi a).
\end{align*}
First observe that for any $\xi \in A^*$,
\begin{align*}
r_1^\sharp (\xi ) = \sum_i \langle \xi, a_i \rangle b_i   \quad \text{ and } \quad r_2^\sharp (\xi ) = \sum_j \langle \xi, x_j \rangle y_j.
\end{align*}
Note that 
\begin{align*}
\langle (\psi \otimes \text{id}_A ) (r_2 ), \xi \otimes \eta \rangle = \sum_{j} \langle \psi (x_j), \xi \rangle \langle y_j , \eta \rangle = \sum_{j} \langle x_j, \psi^* \xi \rangle \langle y_j , \eta \rangle = \langle r_2^\sharp (\psi^* \xi), \eta \rangle
\end{align*}
and
\begin{align*}
\langle (\text{id}_A \otimes \phi ) (r_1), \xi \otimes \eta \rangle = \sum_{i} \langle a_i, \xi \rangle \langle \phi (b_i), \eta \rangle = \langle \phi (\sum_{i} \langle a_i, \xi \rangle b_i), \eta \rangle = \langle \phi r_1^\sharp (\xi), \eta \rangle .
\end{align*}
Hence $\psi \otimes \text{id}_A = \text{id}_A \otimes \phi$. For other parts, we observe that
\begin{align*}
\langle \phi (a) \psi^* (\xi), b \rangle = \langle \psi^* (\xi ), b \cdot \phi (a) \rangle = \langle \xi, \psi (b \cdot \phi(a) ) \rangle  \quad \text{ and }
\end{align*}
\begin{align*}
\langle \psi^* (a \xi ), b \rangle = \langle a \xi, \psi (b) \rangle = \langle \xi, \psi (b) \cdot a \rangle.
\end{align*}
This implies that $\psi ( b \cdot \phi (a) ) = \psi (b) \cdot a$. Similarly,
\begin{align*}
\langle \psi^* (\xi ) \phi (a ), b \rangle = \langle \psi^*(\xi), \phi(a) \cdot b \rangle = \langle \xi, \psi (\phi(a) \cdot b) \rangle ~~~ \quad \text{ and }
\end{align*}
\begin{align*}
\langle \psi^* (\xi a), b \rangle = \langle \xi a, \psi (b) \rangle = \langle \xi, a \cdot \psi (b) \rangle
\end{align*}
which implies that $\psi ( \phi (a) \cdot b ) = a \cdot \psi (b)$.

The converse part is similar.
\end{proof}

One may also define a notion of equivalence between two associative {\bf r}-matrices.
\begin{defn}
Two associative {\bf r}-matrices $r_1, r_2$ on an associative algebra $A$ are said to be equivalent if there exists an algebra isomorphism $\phi : A \rightarrow A$ such that $(\phi \otimes \phi) (r_1) = r_2$. 
\end{defn}

Weak isomorphism is also related to the equivalence of associative {\bf r}-matrices in the following way.

\begin{prop}\label{equi-r}
Two associative {\bf r}-matrices $r_1$ and $r_2$ on an associative algebra $A$ are equivalent if and only if there exists an algebra isomorphism $\phi : A \rightarrow A$ such that $(\phi, \phi^{-1})$ is a weak isomorphism from $r_1$ to $r_2$.
\end{prop}

\begin{proof}
Let $\phi : A \rightarrow A$ defines an equivalence between the {\bf r}-matrices $r_1$ and $r_2$. Then $\phi$ is an algebra morphism and $(\phi \otimes \phi) (r_1) = r_2$. Thus, we have
\begin{align*}
(\phi^{-1} \otimes \text{id} ) (r_2) = (\phi^{-1} \otimes \text{id} ) \circ (\phi \otimes \phi ) (r_1) = (\text{id} \otimes \phi ) (r_1).
\end{align*}
On the other hand, $\phi$ is an algebra map implies
\begin{align*}
\phi (a) \cdot b = \phi ( a \cdot \phi^{-1}(b) )   ~~~~\Rightarrow ~~~~  \phi^{-1} (\phi (a) \cdot b) = a \cdot \phi^{-1} (b).
\end{align*}
Similarly,
\begin{align*}
a \cdot \phi (b) = \phi (\phi^{-1} (a) \cdot b)  ~~~~\Rightarrow ~~~~  \phi^{-1} (a \cdot \phi(b)) = \phi^{-1} (a) \cdot b.
\end{align*}
This shows that $( \phi, \phi^{-1})$ is a weak isomorphism from $r_1$ to $r_2$. Converse part is similar.
\end{proof}

\begin{remark}
The definition of weak isomorphism of two associative {\bf r}-matrices is based on the isomorphism between corresponding $\mathcal{O}$-operators. In general, two weak isomorphic {\bf r}-matrix may not be equivalent.
\end{remark}

Here we define deformations of an associative {\bf r}-matrix by keeping in mind the deformations of $\mathcal{O}$-operator.

\begin{defn}
Let $(A, \cdot)$ be an associative algebra and $r \in \wedge^2 A$ an associative {\bf r}-matrix. An element $\kappa \in \wedge^2 A$ is said to generate a linear deformation of $r$ if for each $t$, the sum $r_t = r + t \kappa$ is an associative {\bf r}-matrix on $A$.
\end{defn}

\begin{defn}
Two linear deformations $r^1_t = r + t \kappa_1$ and $r^2_t = r + t \kappa_2$ of $r$ are said to be equivalent if there exists an element $a \in A$ such that $(\text{id}_A + t (\text{ad}^l_a - \text{ad}^r_a ),~ \text{id}_A + t (\text{ad}^l_a - \text{ad}^r_a ) )$ is a weak homomorphism from $r^1_t$ to $r^2_t$.
\end{defn}

The following result in an easy consequence of the above two definitions.
\begin{prop}
\begin{itemize}
\item[(i)] An element $\kappa \in \wedge^2 A$ generates a linear deformation of $r$ if and only if the map $\kappa^\sharp : A^* \rightarrow A$ generates a linear deformation of the $\mathcal{O}$-operator $r^\sharp$.
\item[(ii)] Two linear deformations $r_t^1 = r + t \kappa_1$ and $r_t^2 = r + t \kappa_2$ are equivalent if and only if $(r_t^1)^\sharp$ and $(r_t^2)^\sharp$ are equivalent linear deformations of the $\mathcal{O}$-operator $r^\sharp$.
\end{itemize}
\end{prop}

One may also define formal deformations of an associative {\bf r}-matrix. Let $r \in \wedge^2 A$ be an associative {\bf r}-matrix. A formal deformation of $r$ consists of a formal sum $r_t = \sum_{i \geq 0} t^i r_i$ satisfying $[[ r_t, r_t ]] = 0$. It is easy to see that  this is equivalent to the fact that $r_t^\sharp = \sum_{i \geq 0} t^i r_i^\sharp$ is a formal deformation of the $\mathcal{O}$-operator $r^\sharp$. Similarly, one may also define equivalence between two formal deformations of $r$ that induce equivalence between the corresponding formal deformations of the $\mathcal{O}$-operator $r^\sharp$.

\medskip

\subsection{Infinitesimal bialgebras}

The notion of infinitesimal bialgebra was introduced by Aguiar as an associative analogue of Lie bialgebras \cite{aguiar2, aguiar3}. Throughout this subsection, we assume that all vector spaces are finite-dimensional so that an associative coalgebra structure on a vector space is equivalent to an associative algebra on its dual space.

Let $(A, \cdot)$ be an associative algebra. Then the tensor product $A \otimes A$ has an $A$-bimodule structure with left and right actions $a(b \otimes c) = (a \cdot b) \otimes c$ and $ ( b \otimes c)a = b \otimes (c \cdot a)$. An infinitesimal bialgebra is an associative algebra $(A, \cdot)$ together with an associative coproduct $\triangle : A \rightarrow A \otimes A$ on $A$ satisfying
\begin{align}\label{inf-bial}
\triangle ( a \cdot b) = a \triangle (b) + \triangle (a) b, ~~~~ \text{ for } a, b \in A.
\end{align}
The condition (\ref{inf-bial}) can be equivalently described by the fact that $\triangle$ is a derivation on $A$ with values in the $A$-bimodule $A \otimes A$.

Let $(A, \cdot_A, \triangle_A)$ and $(B, \cdot_B, \triangle_B)$ be two infinitesimal bialgebras. A morphism between them consists of an algebra morphism $\phi : A \rightarrow B$ which is also compatible with the coproducts in the sense that $(\phi \otimes \phi ) \circ \triangle_A = \triangle_B \circ \phi$. Equivalently, a morphism of infinitesimal bialgebras is a morphism $\phi : A \rightarrow B$ of algebras such that the dual $\phi^* : B^* \rightarrow A^*$ is a morphism of algebras. It is called an isomorphism if $\phi$ is a linear isomorphism.

Let $r \in \wedge^2 A$ be an associative {\bf r}-matrix on an associative algebra $A$. Then the associative algebra structure on $A^*$ (induced from the corresponding $\mathcal{O}$-operator $r^\sharp$) gives rise to an associative coalgebra structure on $A$. We denote this coalgebra structure on $A$ by $\triangle_r$. Moreover, the algebra structure on $A$ and the above coalgebra structure on $A$ forms an infinitesimal bialgebra. Such an infinitesimal bialgebra is called a triangular infinitesimal bialgebra.

Next, we introduce a notion of weak homomorphism between infinitesimal bialgebras whose underlying associative algebra are the same.

\begin{defn}
Let $(A, \cdot, \triangle)$ and $(A, \cdot, \triangle')$ be two infinitesimal bialgebras. A weak homomorphism between them consists of a pair $(\phi, \psi)$ of an algebra morphism $\phi : A \rightarrow A$ and a coalgebra morphism $\psi : A \rightarrow A$ such that
\begin{align*}
\psi ( \phi (a) \cdot b ) = a \cdot \psi (b), \quad \psi ( a \cdot \phi (b) ) = \psi (a) \cdot b.
\end{align*}
\end{defn}

A weak homomorphism $(\phi, \psi)$ is called a weak isomorphism if $\phi, ~ \psi$ are linear isomorphisms.

\begin{prop}\label{iso-bi}
Let $(A, \cdot, \triangle)$ and $(A, \cdot, \triangle')$ be two infinitesimal bialgebras. They are isomorphic as infinitesimal bialgebras if and only if there exists an algebra isomorphism $\phi : A \rightarrow A$ such that the pair $(\phi, \phi^{-1})$ is a weak isomorphism from $(A, \cdot, \triangle)$ to $(A, \cdot, \triangle').$
\end{prop}

In the following, we show that a weak morphism of associative {\bf r}-matrices induces a weak morphism between corresponding infinitesimal bialgebras.


\begin{prop}\label{weak-weak-bi}
Let $(A, \cdot)$ be an associative algebra and $r_1, r_2 \in \wedge^2 A$ be two associative {\bf r}-matrices. If $(\phi, \psi)$ is a weak morphism (resp. weak isomorphism) from $r_1$ to $r_2$, then $(\phi, \psi)$ is a weak morphism (resp. weak isomorphism) from the infinitesimal bialgebra $(A, \cdot, \triangle_{r_1})$ to $(A, \cdot, \triangle_{r_2}).$
\end{prop}

\begin{proof}
Since $(\phi, \psi)$ is a weak morphism from $r_1$ to $r_2$, we already have $\psi (\phi (a) \cdot b) = a \cdot \psi (b)$ and $\psi ( a \cdot \phi (b) ) = \psi (a) \cdot b.$ Thus, it remains to show that $\psi : A \rightarrow A$ is a coalgebra map, or equivalently, $\psi^* : A^* \rightarrow A^*$ is an algebra map. Note that for any $\alpha, \beta \in A^*$ and $x \in A$, we have
\begin{align*}
\langle \psi^* (\alpha \star_{r_1} \beta ), x \rangle 
=~& \langle \alpha \star_{r_1} \beta , \psi (x) \rangle \\
=~& \langle \text{ad}^{* l}_{r_1^\sharp (\alpha)} (\beta)  + \text{ad}^{* r}_{r_1^\sharp (\beta)} (\alpha) , ~ \psi (x) \rangle \\
=~& \langle \beta, \psi (x) \cdot r_1^\sharp (\alpha ) \rangle  + \langle \alpha, r_1^\sharp (\beta) \cdot \psi (x) \rangle \\
=~& \langle \beta , \psi (x \cdot \phi r_1^\sharp (\alpha)) \rangle + \langle \alpha, \psi (\phi r_1^\sharp (\beta) \cdot x ) \rangle \\
=~& \langle \beta , \psi (x \cdot  r_2^\sharp \psi^* (\alpha)) \rangle + \langle \alpha, \psi ( r_2^\sharp \psi^* (\beta) \cdot x ) \rangle \\
=~& \langle \psi^* \beta , x \cdot r_2^\sharp \psi^* (\alpha) \rangle + \langle \psi^* \alpha, r_2^\sharp \psi^* (\beta) \cdot x  \rangle \\
=~& \langle \text{ad}^{* l}_{r_2^\sharp \psi^* \alpha} (\psi^* \beta) , x \rangle + \langle \text{ad}^{* r}_{r_2^\sharp \psi^* \beta} (\psi^* \alpha) , x \rangle \\
=~& \langle \psi^*(\alpha ) \star_{r_2} \psi^* (\beta) , x \rangle.
\end{align*}
This proves that $\psi^* (\alpha \star_{r_1} \beta ) = \psi^*(\alpha ) \star_{r_2} \psi^* (\beta)$. Hence the proof.
\end{proof}

Thus, by combining Propositions \ref{equi-r}, \ref{iso-bi} and \ref{weak-weak-bi}, we have the following. 

\begin{corollary}
Let $r_1, r_2 \in \wedge^2 A$ be two associative {\bf r}-matrices on an associative algebra $(A, \cdot)$. If $r_1, r_2$ are equivalent, then $(A, \cdot, \triangle_{r_1})$ and $(A, \cdot, \triangle_{r_2})$ are isomorphic as infinitesimal bialgebras.
\end{corollary}

Next, we consider linear deformations of an infinitesimal bialgebra by keeping the algebra structure intact. 

Let $(A, \cdot, \triangle)$ be an infinitesimal bialgebra. Consider a map $ \triangle_1 : A \rightarrow A \otimes A$ with the property that for each $t$, the triple $(A, \cdot, \triangle + t \triangle_1)$ is an infinitesimal bialgebra. In this case, we say that $\triangle_1$ generates a linear deformation of the infinitesimal bialgebra $( A, \cdot, \triangle).$ 

The following result can be easily proved using Proposition \ref{r-o-char}.

\begin{prop}
Let $r$ be an associative {\bf r}-matrix on an associative algebra $(A, \cdot)$ with the  infinitesimal bialgebra $(A, \cdot, \triangle_r)$. If $\kappa$ generates a linear deformation of $r$ as an associative {\bf r}-matrix, then $\triangle_\kappa$ (i.e. the dual of the  map $\star_\kappa : A^* \otimes A^* \rightarrow A^*, ~ (\alpha , \beta ) \mapsto \mathrm{ad}_{\kappa^\sharp (\alpha )}^{^*l } \beta ~+~ \mathrm{ad}^{* r}_{\kappa^\sharp (\beta)} \alpha $) generates a linear deformation of the infinitesimal bialgebra $(A, \cdot, \triangle_r).$
\end{prop}

\begin{defn}
Two linear deformations $\triangle_t^1 = \triangle + t \triangle_1$ and $\triangle_t^2 = \triangle + t \triangle_2$ of an infinitesimal bialgebra $(A, \cdot, \triangle )$ are said to be equivalent if there exists an $a \in A$ such that $\big( \mathrm{id}_A + t (\text{ad}^l_a - \text{ad}^r_a), \mathrm{id}_A - t (\mathrm{ad}^l_a - \mathrm{ad}^r_a)    \big)$ is a weak homomorphism from $(A, \cdot, \triangle_t^1)$ to $(A, \cdot, \triangle^2_t).$
\end{defn}

We have the following relation between equivalence of deformations of an associative {\bf r}-matrix and the equivalence of deformations of the corresponding infinitesimal bialgebra.

\begin{prop}
Let $r^1_t = r + t \kappa_1$ and $r_t^2 = r + t \kappa_2$ be two linear deformations of an associative {\bf r}-matrix $r \in \wedge^2 A$. Consider the corresponding linear deformations $(A, \cdot, \triangle_r + t \triangle_{\kappa_1})$ and $(A, \cdot, \triangle_r + t \triangle_{\kappa_2})$ of the corresponding infinitesimal bialgebra $(A, \cdot, \triangle_r)$. Then $r^1_t$ and $r^2_t$ are equivalent if and only if $\triangle^{\kappa_1}_t = \triangle_1 + t \triangle_{\kappa_1}$ and $\triangle^{\kappa_2}_t = \triangle + t \triangle_{\kappa_2}$ are equivalent.
\end{prop}

Note that in the above deformation of an infinitesimal bialgebra, we only deform the coalgebra structure. This perfectly fits with the deformation of associative {\bf r}-matrices and deformations of $\mathcal{O}$-operators as in these cases also, we only deform the {\bf r}-matrix or the $\mathcal{O}$-operator keeping the underlying algebra un-deformed. Therefore, it would be interesting to consider more general deformation of $\mathcal{O}$-operators, associative {\bf r}-matrices by allowing deformation of the underlying algebra.

\subsection{Averaging operators} Averaging operators in the commutative algebra context was first introduced by Reynolds in turbulence theory around 100 years ago \cite{rey}. The time average of a real-valued function $f$ defined on time-space is an averaging operator. Such operators on $C(X)$ (resp. $C_\infty (X)$, the algebra of scalar-valued continuous functions (resp. vanishing at infinity) on a compact Hausdorff space $X$ was studied in \cite{ kelley, rota2}.

Recently, averaging operators has also been studied in associative context \cite{pei-guo}. Let $A$ be an associative algebra. A linear map $P : A \rightarrow A$ is said to be a left averaging operator if it satisfies
\begin{align*}
P (a) \cdot P(b) = P ( P(a) \cdot b), ~~~~ \text{ for all } a, b \in A.
\end{align*}
A right averaging operator is a linear map $P : A \rightarrow A$ that satisfies $P (a) \cdot P(b) = P ( a \cdot P(b)), ~~~~ \text{ for all } a, b \in A.$ Finally, an averaging operator on $A$ is a linear map $P : A \rightarrow A$ which is both a left and right averaging operator. In this paper, we only mention the deformations of left or right averaging operators. Deformations of full averaging operators will be discussed in an upcoming paper \cite{das-av}.

Note that a left (resp. right) averaging operator on an associative algebra $A$ can be thought of as an $\mathcal{O}$-operator on $A$ with respect to the bimodule $(A, \text{ad}^l, 0)$ (resp.   $(A, 0,  \text{ad}^r)$). Thus, following the results of Section \ref{sec2}, we obtain the following.

\begin{thm}
Let $A$ be an associative algebra.
\begin{itemize}
\item[(i)] There exists a graded Lie algebra structure on $C^\bullet (A, A ) = \oplus_{n \geq 0} \mathrm{Hom} (A^{\otimes n}, A)$ where Maurer-Cartan elements are exactly left averaging operators on $A$.
\item[(ii)] There exists a graded Lie algebra structure (different from the previous one) on $C^\bullet (A, A ) = \oplus_{n \geq 0} \mathrm{Hom} (A^{\otimes n}, A)$ where Maurer-Cartan elements are right averaging operators on $A$.
\end{itemize}
\end{thm}

The cohomology induced from the left averaging operator $P$ (viewed as a Maurer-Cartan element) is called the cohomology of $P$. The cohomology of a right averaging operator can be defined in a similar way.

A formal deformation of a left averaging operator $P$ is a sum $P_t = \sum_{i \geq 0} t^i P_i$ with $P_0 = P$ such that $P_t$ is a left averaging operator on $A[[t]]$. In other words, 
\begin{align*}
P_t (a) \cdot P_t(b) = P_t ( P_t(a) \cdot b), ~~~~ \text{ for all } a, b \in A.
\end{align*}

One may also define linear deformations, Nijenhuis elements for a left averaging operator. Nijenhuis elements then define trivial linear deformations. The vanishing of the second cohomology of a left averaging operator ensures the extension of a finite order deformation to the next order. Deformations of right averaging operators and the corresponding results are analogous.

\vspace*{0.5cm}

\noindent {\bf Acknowledgements.} The author thanks the referee for his/her comments on the earlier version of the manuscript. The research is supported by the post-doctoral fellowship of Indian Institute of Technology (IIT) Kanpur. The author thanks the Institute for support.

\end{document}